\documentclass[10pt,reqno]{amsart}
\usepackage{amssymb,mathrsfs,graphicx}
\usepackage{ifthen}
\usepackage{caption}
\usepackage{sidecap}
\usepackage{rotating}

\usepackage{colortbl}
\definecolor{black}{rgb}{0.0, 0.0, 0.0}
\definecolor{red}{rgb}{1.0, 0.5, 0.5}
\provideboolean{shownotes} 
\setboolean{shownotes}{true} 
\newcommand{\margnote}[1]{
\ifthenelse{\boolean{shownotes}}%
{\marginpar{\raggedright\tiny\texttt{#1}}}%
{}%
}
\newcommand{\hole}[1]{
\ifthenelse{\boolean{shownotes}}%
{\begin{center} \fbox{ \rule {.25cm}{0cm} \rule[-.1cm]{0cm}{.4cm}
\parbox{.85\textwidth}{\begin{center} \texttt{#1}\end{center}} \rule
{.25cm}{0cm}}\end{center}} {} }

\graphicspath{{pics/}}

\title[Quantitative error estimates for the large friction limit]{Quantitative error estimates for the large friction limit of Vlasov equation with nonlocal forces}

\author[Carrillo]{Jos\'{e} A. Carrillo}
\address[Jos\'{e} A. Carrillo]{\newline Department of Mathematics
    \newline Imperial College London, London SW7 2AZ, United Kingdom}
\email{carrillo@imperial.ac.uk}

\author[Choi]{Young-Pil Choi}
\address[Young-Pil Choi]{\newline Department of Mathematics\newline
Yonsei University, Seoul 03722, Republic of Korea}
\email{ypchoi@yonsei.ac.kr}

\numberwithin{equation}{section}

\newtheorem{theorem}{Theorem}[section]
\newtheorem{lemma}{Lemma}[section]

\newtheorem{proposition}{Proposition}[section]
\newtheorem{remark}{Remark}[section]
\newtheorem{definition}{Definition}[section]

\newcommand{\R}{\mathbb R}

\newcommand{\ls}{\lesssim}

\newcommand{\mc}{\mathcal C}

\newcommand{\bq}{\begin{equation}}
\newcommand{\eq}{\end{equation}}
\newcommand{\e}{\varepsilon}
\newcommand{\lt}{\left}
\newcommand{\rt}{\right}
\newcommand{\lal}{\langle}
\newcommand{\ral}{\rangle}
\newcommand{\pa}{\partial}

\newcommand{\mh}{\mathcal{H}}

\newcommand{\mf}{\mathcal{F}}
\newcommand{\W}{\mathcal{W}}
\newcommand{\md}{\mathcal{D}}
\newcommand{\mt}{\mathcal{T}}

\begin{document}
\allowdisplaybreaks

\date{\today}




\begin{abstract} We study an asymptotic limit of Vlasov type equation with nonlocal interaction forces where the friction terms are dominant. We provide a quantitative estimate of this large friction limit from the kinetic equation to a continuity type equation with a nonlocal velocity field, the so-called aggregation equation, by employing $2$-Wasserstein distance. By introducing an intermediate system, given by the pressureless Euler equations with nonlocal forces, we can quantify the error between the spatial densities of the kinetic equation and the pressureless Euler system by means of relative entropy type arguments combined with the $2$-Wasserstein distance. This together with the quantitative error estimate between the pressureless Euler system and the aggregation equation in $2$-Wasserstein distance in [Commun. Math. Phys, 365, (2019), 329--361] establishes the quantitative bounds on the error between the kinetic equation and the aggregation equation.

\end{abstract}

\maketitle \centerline{\date}


%
%
%
%
\section{Introduction}
Let $f = f(x,v,t)$ be the particle distribution function at $(x,v) \in \R^d \times \R^d$ and at time $t \in \R_+$ for the following kinetic equation:
\bq\label{main_eq}
\pa_t f + v\cdot\nabla_x f -  \nabla_v \cdot \lt((\gamma v + \lambda\lt(\nabla_x V + \nabla_x W \star \rho\rt))f \rt) = \nabla_v \cdot (\beta (v - u)f) 
\eq
subject to the initial data
\[
f(x,v,t)|_{t=0} =: f_0(x,v), \quad (x,v) \in \R^d \times \R^d,
\]
where $u$ is the local particle velocity, i.e., 
$$
u = \frac{ \int_{\R^d} vf\,dv} {\rho}\quad \mbox{ with } \rho := \int_{\R^d} f\,dv\,,
$$
$V$ and $W$ are the confinement and the interaction potentials, respectively. In \eqref{main_eq}, the first two terms take into account the free transport of the particles, and the third term consists of linear damping with a strength $\gamma>0$ and the particle confinement and interaction forces in position due to the potentials with strength $\lambda>0$. The right hand side of \eqref{main_eq} is the local alignment force for particles as introduced in \cite{KMT13} for swarming models. In fact, it can also be understood as the localized version of the nonlinear damping term introduced in \cite{MT} as a suitable normalization of the Cucker-Smale model \cite{CS}.
Notice that this alignment term is also a nonlinear damping relaxation towards the local velocity used in classical kinetic theory \cite{CIP,Vilkinetic}. Throughout this paper, we assume that $f$ is a probability density, i.e., $\|f(\cdot,\cdot,t)\|_{L^1} = 1$ for $t\geq 0$, since the total mass is preserved in time. 

In the current work, we are interested in the asymptotic analysis of \eqref{main_eq} when considering singular parameters. More specifically, we deal with the large friction limit to a continuity type equation from the kinetic equation \eqref{main_eq} when the parameters $\gamma, \lambda > 0$, and $\beta > 0$ get large enough. Computing the moments on the kinetic equation \eqref{main_eq}, we find that the local density $\rho$ and local velocity  $u$ satisfy
$$\begin{aligned}
&\pa_t \rho + \nabla_x \cdot (\rho u) = 0,\cr
&\pa_t (\rho u) + \nabla_x \cdot (\rho u \otimes u) + \nabla_x \cdot \lt( \int_{\R^d} (v-u)\otimes (v-u) f(x,v,t)\,dv\rt) \cr
&\hspace{2.5cm} = - \gamma \rho u - \lambda\rho(\nabla_x V + \nabla_x W \star \rho).
\end{aligned}$$
As usual, the moment system is not closed. By letting the friction of the equation \eqref{main_eq} very strong, i.e., $\gamma, \lambda, \beta  \gg 1$, for  instance, $\gamma=\lambda=\beta=o\left(\e^{-1}\right) \to + \infty$ with $\lambda/\gamma=o(1)\to \kappa>0$ as $\e \to 0$, then at the formal level, we find
$$
 \nabla_v \cdot \lt((2v -u + \kappa \lt(\nabla_x V + \nabla_x W \star \rho\rt))f \rt) =0\,,
$$
and thus,
\[
f(x,v,t) \simeq \rho(x,t) \otimes \delta_{v - u(x,t)} \quad \mbox{and} \quad \rho u \simeq -\kappa\rho(\nabla_x V + \nabla_x W \star \rho) \quad \mbox{for} \quad \e \ll 1
\]
is the element in its kernel with the initial monokinetic distribution $\rho(x,0) \otimes \delta_{v - u(x,0)}$.

Those relations provide that the density $\rho$ satisfies the following continuity type equation with a nonlocal velocity field, the so-called {\it aggregation equation}, see for instance \cite{BCL,BLR,CCH}
and the references therein,
\bq\label{main_conti}
\pa_t \rho + \nabla_x \cdot (\rho u) = 0, \quad \rho u = -\kappa\rho\lt(\nabla_x V + \nabla_x W \star \rho\rt).
\eq

The large friction limit has been considered in \cite{Jabin00}, where the macroscopic limit of a Vlasov type equation with friction is studied by using a PDE approach, and later the restrictions on the functional spaces for the solutions and the conditions for interaction potentials are relaxed in \cite{FS15} by employing PDE analysis and the method of characteristics. More recently, these results have been extended  in \cite{FST16} for more general Vlasov type equations; Vlasov type equations with nonlocal interaction and nonlocal velocity alignment forces. However, all of these results in \cite{FS15,FST16,Jabin00} are based on compactness arguments, and to our best knowledge, quantitative estimates for the large friction limit have not yet been obtained. The large friction limit has received a lot of attention at the hydrodynamic level by the conservation laws community, see for instance \cite{CG,MM,LC,GLT,LT}, but due to their inherent difficulties, it has been elusive at the kinetic level.

The main purpose of this work is to render the above formal limit to the nonlocal aggregation equation completely rigorous with quantitative bounds. Our strategy of the proof uses an intermediate system to divide the error estimates as depicted in Figure \ref{schemeofproof}. 
\begin{figure}[!ht]
\centering
\includegraphics[width=1\textwidth,clip]{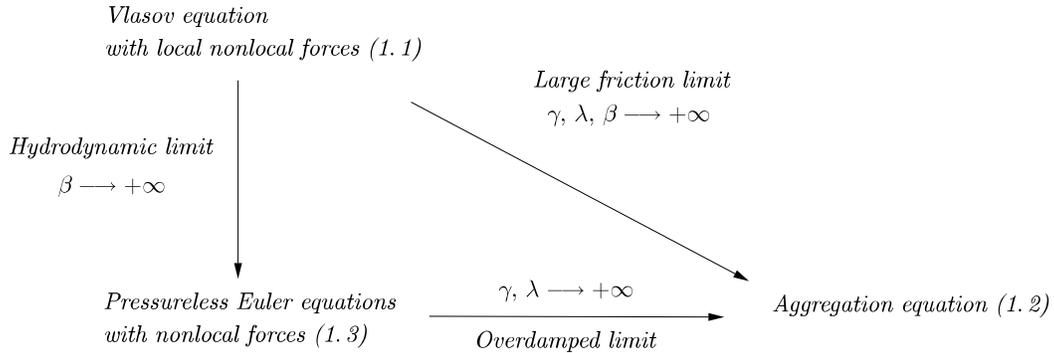}
\caption{
Schematic illustration of the strategy of the proof.
}
\label{schemeofproof}
\end{figure}
We first fix $\lambda$ and $\gamma$ with $\kappa \gamma = \lambda$ and take $\beta = 1/\e$. We denote by $f^{\gamma,\e}$ the solution to the associated kinetic equation \eqref{main_eq}. We then introduce an intermediate system, given by the pressureless Euler equations with nonlocal interactions, between the kinetic equation \eqref{main_eq} and the limiting equation \eqref{main_conti}:
\begin{align}\label{main_eq2}
\begin{aligned}
&\pa_t \rho^{\gamma} + \nabla_x \cdot (\rho^{\gamma} u^{\gamma}) = 0,\cr
&\pa_t (\rho^{\gamma} u^{\gamma}) + \nabla_x \cdot (\rho^{\gamma} u^{\gamma} \otimes u^{\gamma}) =  - \gamma \lt(\rho^{\gamma} u^{\gamma} + \kappa\rho^{\gamma}(\nabla_x V + \nabla_x W \star \rho^{\gamma})\rt).
\end{aligned}
\end{align}
In order to estimate the error between two solutions $\rho^{\gamma,\e}$ and $\rho^{\gamma}$ to \eqref{main_eq} and \eqref{main_eq2}, respectively, where 
$$
\rho^{\gamma,\e} := \int_{\R^d} f^{\gamma,\e}\,dv\,,
$$
we use the Wasserstein distance which is defined by
\[
W_p(\mu, \nu):= \inf_{\pi \in \Pi(\mu,\nu)}\lt(\int_{\R^d \times \R^d} |x - y|^p d\pi(x,y) \rt)^{1/p}
\]
for $p\geq 1$ and $\mu,\nu \in \mathcal{P}_p(\R^d)$, where $\Pi(\mu,\nu)$ is the set of all probability measures on $\R^d \times \R^d$ with first and second marginals $\mu$ and $\nu$ and bounded $p$-moments, respectively. We refer to \cite{AGS08,Vil08} for discussions of various topics related to the Wasserstein distance.

Employing the $2$-Wasserstein distance, we first obtain the quantitative estimate for $W_2^2(\rho^{\gamma,\e},\rho^{\gamma})$ with the aid of the relative entropy argument. It is worth mentioning that the entropy for the system \eqref{main_eq2} is not strictly convex with respect to $\rho$ due to the absence of pressure in the system, see Section \ref{sec_re} for more details. Thus the relative entropy estimate is not enough to provide the error estimates between the spatial density $\rho^{\gamma,\e}$ and the density $\rho^{\gamma}$. We also want to emphasize that the relative entropy estimate is even not closed due to the nonlinearity and nonlocality of the interaction term $\nabla_x W \star \rho$. We provide a new inequality which gives a remarkable relation between the $2$-Wasserstein distance and the relative entropy, see Lemma \ref{prop_rho_wa}. Using that new observation together with combining the relative entropy estimate and the $2$-Wasserstein distance between the solutions in a hypocoercivity type argument, we have the quantitative error estimate for the vertical part of the diagram in Figure \ref{schemeofproof}. Let us point out that in order to make this step rigorous, we need to work with strong solutions to the pressureless Euler system \eqref{main_eq2} for two reasons. On one hand, strong solutions are needed for making sense of the integration by parts required for the relative entropy argument. On the other hand, some regularity on the velocity field, the boundedness of the spatial derivatives of the velocity field uniformly in $\gamma$, is needed in order to control terms appearing due to the time derivatives of $W_2^2(\rho^{\gamma,\e},\rho^{\gamma})$ and the relative entropy.

We finally remark that the closest result in the literature to ours is due to Figalli and Kang in \cite{FK19}. It concerns with the vertical part of the diagram in Figure \ref{schemeofproof} for a related system without interaction forces but Cucker-Smale alignment terms. Even if they already combined the $2$-Wasserstein distance and the relative entropy between $\rho^{\gamma,\e}$ and $\rho^{\gamma}$, they did not take full advantage of the  $2$-Wasserstein distance, see Remark \ref{rmk_hydro} for more details. This is our main contribution in this step.

The final step, corresponding to the bottom part of the diagram in Figure \ref{schemeofproof}, is inspired on a recent work of part of the authors \cite{CCTpre}. Actually, we can estimate the error between the solutions $\rho^{\gamma}$ and $\rho$ to \eqref{main_eq2} and \eqref{main_conti}, respectively, in the $2$-Wasserstein distance again. Here, it is again crucial to use the boundedness of the spatial derivatives of the velocity field uniformly in $\gamma$. Combining the above arguments, we finally conclude the main result of our work: the quantitative error estimate between two solutions $\rho^{\gamma,\e}$ and $\rho$ to the equations \eqref{main_eq} and \eqref{main_conti}, respectively, in the $2$-Wasserstein distance.

Before writing our main result, we remind the reader of a well known estimate for the total energy of the kinetic equation \eqref{main_eq}. For this, we define the total energy $\mf$ and the associated dissipations $\md_1$  and $\md_2$ as follows: 
\begin{equation}\label{freen}
\mf(f) := \frac12\int_{\R^d \times \R^d} |v|^2 f\,dxdv + \frac{\lambda}{2}\int_{\R^d \times \R^d} W(x-y)\rho(x) \rho(y)\,dxdy + \lambda \int_{\R^d} V \rho\,dx, 
\end{equation}
$$\begin{aligned}
\md_1(f) &:= \int_{\R^d \times \R^d} f\lt|u-v\rt|^2 dxdv, \quad \mbox{and} \quad \md_2(f):= \int_{\R^d \times \R^d} |v|^2 f\,dxdv,
\end{aligned}$$
respectively. Suppose that $f$ is a solution of \eqref{main_eq} with sufficient integrability, then it is straightforward to check that
\begin{equation}\label{lem_energy}
\frac{d}{dt}\mf(f) + \beta D_1(f) + \gamma D_2(f) = 0.
\end{equation}
Notice that weak solutions may only satisfy an inequality in the above relation that is enough for our purposes. 

In order to control the velocity field for the intermediate pressureless Euler equations \eqref{main_eq2}, we assume that the confinement potential $V$ and the interaction potential $W$ satisfy:
\begin{itemize}
\item[{\bf (H)}] The confinement potential $V(x) = c_V|x|^2/2$, and the interaction potential $W$ satisfies $W(-x) = W(x)$, $\nabla_x W \in (\W^{1,\infty} \cap \W^{[d/2]+1,\infty})(\R^d)$, and $c_V + c_W > 0$ with 
\[
c_W := \inf_{x \neq y} \frac{\lt\lal x-y, \nabla_x W(x) - \nabla_y W(y) \rt\ral}{|x-y|^2}. 
\]
\end{itemize}
We are now in position to state the main result of this work. 

\begin{theorem}\label{thm_main} Assume that initial data $f_0^\e$ satisfy
\[
\sup_{\e > 0}\|(1+ |v|^2 + V)f_0^\e\|_{L^1} < \infty, \quad  f_0^\e \in L^\infty(\R^d \times \R^d), \quad \mbox{and} \quad \rho^\e_0(W\star \rho^\e_0) \in L^1(\R^d)
\]
for all $\e > 0$. Let $f^\e$ be a solution to the equation \eqref{main_eq} with $\beta = 1/\e$, $\kappa\gamma = \lambda =1/\e$ with $\kappa>0$ up to time $T>0$, such that $f^\e \in L^\infty(0,T; (L^1\cap L^\infty)(\R^d \times \R^d))$ satisfying the energy inequality \eqref{lem_energy} with initial data $f_0^\e$. Let $\rho$ be a solution of \eqref{main_conti} up to the time $T$, such that $\rho \in \mc([0,T];\mathcal{P}_2(\R^d))$ with initial data $\rho_0$ satisfying
\[
\rho_0 \in \mathcal{P}_2(\R^d) \quad \mbox{and} \quad \int_{\R^d} \lt(|u_0|^2 + V+W\star \rho_0\rt)\rho_0\,dx < \infty.
\] 
Suppose that ${\bf (H)}$ holds. Then, for $\e,\kappa > 0$ small enough, we have the following quantitative bound:
\[
\int_0^T W_2^2(\rho^\e(t), \rho(t))\,dt \leq \mathcal{O}(\e) + CW_2^2(\rho^\e_0,\rho_0),
\]
where $\rho^\e = \int_{\R^d} f^\e\,dv$ and $C > 0$ is independent of $\e$.
\end{theorem}

\begin{remark} As mentioned above, our strategy consists in using  \eqref{main_eq2} as intermediate system and compare the errors from the kinetic equation \eqref{main_eq} to the pressureless Euler equations \eqref{main_eq2} and from \eqref{main_eq2} to the aggregation equation \eqref{main_conti}. These estimates hold as long as there exist strong solutions to the system \eqref{main_eq2} up to the given time $T>0$. Strong solutions can be obtained locally in time by only assuming $\nabla_x W \in (\W^{1,1} \cap \W^{1,\infty})(\R^d)$, see Theorem \ref{thm_local}. However, in order to ensure existence on any arbitrarily large time interval $[0,T)$, the additional regularity for $\nabla_x W$ is required, see Theorem \ref{thm_glo}. Moreover, our error estimates in Section \ref{sec_quanti} and Section \ref{sec_lft} only need the regularity $\nabla_x W \in (\W^{1,1} \cap \W^{1,\infty})(\R^d)$ too.
\end{remark}

The rest of paper is organized as follows. In Section \ref{sec_quanti}, we provide a quantitative error estimate the kinetic equation \eqref{main_eq} and the intermediate pressureless Euler system with nonlocal forces \eqref{main_eq2} by means of the relative entropy argument together with $2$-Wasserstein distance. Section \ref{sec_lft} is devoted to give the details of the proof for our main result on the large friction limit, and the required global-in-time existence theories for the equations \eqref{main_eq}, \eqref{main_conti}, and \eqref{main_eq2} are presented in Section \ref{sec_ext}.

%
%
%
%
\section{Quantitative error estimate between \eqref{main_eq} and \eqref{main_eq2}}\label{sec_quanti}

In this section, we provide the quantitative error estimate between weak solutions to the kinetic equation \eqref{main_eq} and a unique strong solution to the system \eqref{main_eq2} by employing the relative entropy estimate together with $2$-Wasserstein distance. As mentioned in Introduction, we estimate the $2$-Wasserstein distance between the spatial density of \eqref{main_eq} and the density of \eqref{main_eq2}. This together with the standard relative entropy estimate gives our desired quantitative estimate. Note that in this section the result allows more general potentials $V$ and $W$; the particular choice $V = c_V|x|^2/2$ is not required, and the condition $c_V + c_W > 0$ appeared in {\bf (H)} is not needed. The assumption {\bf (H)} implies that the sum of the last two terms in \eqref{freen} related to the macroscopic density $\rho$ involving $V$ and $W$ in the total energy $\mathcal{F}$ is displacement convex with respect to $2$-Wasserstein distance. This fact will be used for the estimate of the large friction limit from \eqref{main_eq2} to \eqref{main_conti} in Section \ref{sec_lft}.

For notational simplicity, we drop the $\gamma$-dependence in solutions and denote by $f^\e := f^{\gamma,\e}, \rho:= \rho^{\gamma}, u:=u^{\gamma}$ throughout this section. In the following two subsections, we prove the proposition below on the quantitative estimate of $2$-Wasserstein distance between solutions to \eqref{main_eq} and \eqref{main_eq2}.
\begin{proposition}\label{prop_main} Let $f^\e$ be the solution to the equation \eqref{main_eq} and $(\bar\rho,\bar u)$ be the strong solution to the system \eqref{main_eq2} on the time interval $[0,T]$. 
Suppose that $\gamma > 0$ is large enough such that $\gamma -C\lambda - e^{C_{\bar u}}(1+\lambda) > 0$, where $C_{\bar u} := C\|\nabla_x \bar u\|_{L^\infty(0,T;L^\infty)}$. Furthermore, we assume that the confinement potential $V$ is bounded from below and the interaction potential $W$ is symmetric and $\nabla_x W \in \W^{1,\infty}(\R^d)$. 
Then we have
\[
W_2^2(\rho^\e(t), \bar\rho(t)) \leq e^{C_{\bar u}}\lt( W_2^2(\rho^\e_0, \bar\rho_0) + \frac{\mathcal{I}(U^\e_0, \bar U_0) + C_{\bar u}\max\{1,\lambda\}\e + e^{C_{\bar u}}\lambda W_2^2(\rho^\e_0,\bar\rho_0)}{\gamma -C\lambda - e^{C_{\bar u}}(1+\lambda)} \rt),
\]
where $\mathcal{I}(U^\e_0, U_0)$ is given by
\[
\mathcal{I}(U^\e_0, \bar U_0) := \int_{\R^d}  \rho^\e_0(x)| u^\e_0(x) - \bar u_0(x)|^2\,dx + \int_{\R^d} \lt( \int_{\R^d} f_0^\e |v|^2\,dv  - \bar \rho_0|\bar u_0|^2\rt)dx,
\]
and $C>0$ is independent of $\gamma, \lambda$ and $\e$, but depends on $T$. \end{proposition}

\begin{remark}Without loss of generality, we assume that $V \geq 0$ in the rest of this section.
\end{remark}

%
%
%
%
\subsection{Relative entropy estimate}\label{sec_re}
We rewrite the equations \eqref{main_eq2} in conservative form:
\[
U_t + \nabla_x \cdot A(U) = F(U),
\qquad
\mbox{where } 
m := \rho u, \quad U := \begin{pmatrix}
\rho \\
m 
\end{pmatrix},
\quad
A(U) := \begin{pmatrix}
m  \\
\frac{m \otimes m}{\rho} 
\end{pmatrix},
\]
and
\[
F(U) := -\begin{pmatrix}
0 \\
\displaystyle \gamma \rho u + \lambda \rho\lt(\nabla_x V + \nabla_x W \star \rho \rt)
\end{pmatrix}.
\]
Then the above system has the following macro entropy form $E(U) := \tfrac{|m|^2}{2\rho}$.
Note that the entropy defined above is not strictly convex with respect to $\rho$. 
We now define the relative entropy functional $\mh$ as follows.
\bq\label{def_rel}
\mh(U| \bar U) := E( U) - E(\bar U) - DE(\bar U)( U-\bar U) \quad \mbox{with} \quad \bar U := \begin{pmatrix}
        \bar\rho \\
        \bar m \\
    \end{pmatrix},
\eq
where $D E(U)$ denotes the derivation of $E$ with respect to $\rho, m$, i.e.,
\[
DE(U) = \begin{pmatrix}
\displaystyle        -\frac{|m|^2}{2\rho^2} \\[3mm]
\displaystyle        \frac{m}{\rho} 
    \end{pmatrix}.
\]
This yields 
\[
\mh(U|\bar U) = \frac{\rho|u|^2}{2} - \frac{\bar \rho|\bar u|^2}{2} - \frac{|\bar u|^2}{2}(\bar \rho - \rho) - \bar u\cdot (\rho u - \bar\rho \bar u)= \frac{\rho}{2}|u - \bar u|^2.
\]
We next derive an evolution equation for the integrand relative entropy in the lemma below.
\begin{lemma}\label{lem_rel}The relative entropy $\mh$ defined in \eqref{def_rel} satisfies the following equality:
\begin{align*}
\begin{aligned}
\frac{d}{dt}\int_{\R^d} \mh(U|\bar U)\,dx &= \int_{\R^d} \partial_t E(U)\,dx - \int_{\R^d} \nabla_x (DE(\bar U)):A(U|\bar U)\,dx\cr
&\quad - \int_{\R^d} DE(\bar U)\left[ \pa_t U + \nabla_x \cdot A(U) - F(U)\right]dx\cr
&\quad -\gamma \int_{\R^d} \rho| \bar u - u|^2 - \rho |u|^2\,dx + \lambda \int_{\R^d} \nabla_x V \cdot \rho u\,dx\cr
&\quad + \lambda \int_{\R^d} \rho (u - \bar u) \cdot \nabla_x W \star (\bar \rho - \rho) + \rho u \cdot \nabla_x W \star \rho\,dx,
\end{aligned}
\end{align*}
where $A(U|\bar U) := A(U) - A(\bar U) - DA(\bar U)(U-\bar U)$ is the relative flux functional.
\end{lemma}
\begin{proof}It follows from \eqref{def_rel} that
\begin{align*}
\begin{aligned}
\frac{d}{dt}\int_{\R^d} \mh(U|\bar U)\,dx & = \int_{\R^d} \partial_t E(U)\,dx - \int_{\R^d} DE(\bar U)(\pa_t U + \nabla_x \cdot A(U)- F(U))\,dx \cr
&\quad +\int_{\R^d} D^2 E(\bar U) \nabla_x \cdot A(\bar U)(U-\bar U) + DE(\bar U) \nabla_x \cdot A(U)\,dx\cr
&\quad -\int_{\R^d} D^2 E(\bar U)F(\bar U)(U-\bar U) + DE(\bar U)F(U)\,dx\cr
&=: \sum_{i=1}^4 I_i.
\end{aligned}
\end{align*}
Integrating by parts, the following identity holds
\[
\int_{\R^d} D^2 E(\bar U) : \nabla_x \cdot A(\bar U) (U - \bar U)\,dx = \int_{\R^d} \nabla_x DE(\bar U) : DA(\bar U)(U - \bar U)\,dx,
\]
see \cite[Lemma 4.1]{KMT15} for details of proof. Moreover, we also find from \cite{KMT15} that
\[
\int_{\R^d} \nabla_x DE(\bar U): A(\bar U)\,dx = 0.
\]
Thus we obtain
\begin{align*}
\begin{aligned}
I_3 &= \int_{\R^d} \left(\nabla_x DE(\bar U)\right):\left( DA(\bar U)(U-\bar U) - A(U)\right)dx  \cr
&= -\int_{\R^d} \left(\nabla_x DE(\bar U)\right):\left(A(U|\bar U) + A(\bar U)\right)dx\cr
&=-\int_{\R^d} \left(\nabla_x DE(\bar U)\right):A(U|\bar U)\,dx.
\end{aligned}
\end{align*}
For the estimate $I_4$, we notice that
\begin{equation*}
    DE(\bar U) = \begin{pmatrix}
\displaystyle       -\frac{|\bar m|^2}{2\bar\rho^2}  \\[4mm]
\displaystyle        \frac{\bar m}{\bar\rho}
    \end{pmatrix}
    \quad \mbox{and} \quad
    D^2E(\bar U) = \begin{pmatrix}
    * & \displaystyle - \frac{\bar m}{\bar \rho^2}  \\[4mm]
        * & \displaystyle \frac{1}{\bar \rho} 
    \end{pmatrix}.
\end{equation*}
Then, by a direct calculation, we find
\[
D^2 E(\bar U)F(\bar U)(U-\bar U) = -\rho(x) (u(x) - \bar u(x))\cdot \lt(\gamma \bar u + \lambda \nabla_x V + \lambda \nabla_x W \star \bar\rho \rt) 
\]
and
\[
DE(\bar U)F(U) = - \rho \bar u \cdot \lt(\gamma u + \lambda \nabla_x V + \lambda \nabla_x W \star \rho \rt).
\]
Thus we obtain
$$\begin{aligned}
-I_4 &= -\int_{\R^d}  \rho(x) (u(x) - \bar u(x))\cdot\lt(\gamma \bar u(x) + \lambda(\nabla_x V(x) + (\nabla_x W \star \bar\rho)(x)) \rt) \,dx \cr
&\quad - \int_{\R^d}  \rho(x) \bar u(x)\cdot\lt(\gamma u(x)+ \lambda \nabla_x V(x) + \lambda (\nabla_x W \star \rho)(x) \rt) \,dx \cr
&= \gamma \int_{\R^d} \rho| \bar u - u|^2 - \rho |u|^2\,dx - \lambda \int_{\R^d} \nabla_x V \cdot  \rho  u\,dx\cr
&\quad - \lambda \int_{\R^d} \rho (u - \bar u) \cdot \nabla_x W \star (\bar \rho - \rho) + \rho u \cdot \nabla_x W \star \rho\,dx.
\end{aligned}$$
Combining the above estimates concludes the desired result.
\end{proof}

In the light of the previous lemma, we provide the following proposition.
\begin{proposition}\label{prop_re}Let $f^\e$ be the solution to the equation \eqref{main_eq} and $(\bar \rho,\bar u)$ be the strong solution to the system \eqref{main_eq2} on the time interval $[0,T]$. Then we have
\begin{align}\label{eqn_rel}
\begin{aligned}
&\int_{\R^d} \mh(U^\e(t)|\bar U(t))\,dx + \gamma\int_0^t \int_{\R^d} \rho^\e(x)| u^\e(x) - \bar u(x)|^2\,dxds\cr
&\qquad \leq \int_{\R^d} \mh(U^\e_0|\bar U_0)\,dx + \int_{\R^d} \lt( \int_{\R^d} f_0^\e |v|^2\,dv  - \bar \rho_0|\bar u_0|^2\rt)dx + C\|\nabla_x \bar u\|_{L^\infty}\max\{1,\lambda\} \e \cr
&\qquad \quad + C\|\nabla_x \bar u\|_{L^\infty}\int_0^t \int_{\R^d} \mh(U^\e(s)|\bar U(s))\,dxds  \cr
&\qquad \quad + \lambda\int_0^t \int_{\R^d} \rho^\e(x) ( u^\e(x) - \bar u(x)) \cdot (\nabla_x W \star (\bar \rho - \rho^\e))(x)\,dxds.
\end{aligned}
\end{align}
\end{proposition}
\begin{proof}It follows from Lemma \ref{lem_rel} that
\begin{align*}
\begin{aligned}
&\int_{\R^d} \mh(U^\e(t)|\bar U(t))\,dx + \gamma\int_0^t \int_{\R^d} \rho^\e(x)| u^\e(x) - \bar u(x)|^2\,dxds\cr 
&\quad = \int_{\R^d} \mh(U_0^\e|\bar U_0)\,dx + \int_{\R^d} E(U^\e) - E(\bar U_0)\,dx - \int_0^t \int_{\R^d} \nabla_x (DE(\bar U)):A(U^\e|\bar U)\,dxds \cr
&\qquad - \int_0^t \int_{\R^d} DE(\bar U)\left[ \pa_s U^\e + \nabla_x \cdot A(U^\e) - F(U^\e)\right]dxds \cr
&\qquad + \gamma\int_0^t \int_{\R^d} \rho^\e(x) |u^\e(x)|^2\,dxds + \lambda \int_0^t \int_{\R^d} \nabla_x V(x) \cdot \rho^\e(x) u^\e(x)\,dxds\cr
&\qquad + \lambda\int_0^t \int_{\R^d} \rho^\e(x) ( u^\e(x) - \bar u(x)) \cdot (\nabla_x W \star (\bar \rho - \rho^\e))(x) +  \rho^\e(x) u^\e(x) \cdot (\nabla_x W \star \rho^\e)(x)\,dxds\cr
&\quad =: \sum_{i=1}^7 J_i^\e.
\end{aligned}
\end{align*}
Here $J_i^\e, i =2,\cdots,7$ can be estimated as follows. \newline

\noindent {\bf Estimate of $J_2^\e$}: Note that
\bq\label{est_uf}
|u^\e|^2 =  \lt|\frac{\displaystyle \int_{\R^d} vf^\e\,dv}{\displaystyle\int_{\R^d} f^\e\,dv} \rt|^2 \leq \frac{\displaystyle \int_{\R^d} |v|^2 f^\e\,dv}{\rho^\e}, \quad \mbox{i.e.,} \quad \rho^\e|u^\e|^2 \leq \int_{\R^d} |v|^2 f^\e\,dv.
\eq
This gives
\[
E(U^\e) = \frac12 \rho^\e|u^\e|^2\,  \leq \frac12\int_{\R^d} |v|^2 f^\e\,dv =: K(f^\e).
\]
Thus, by adding and subtracting the functional $K(f^\e)$, we find
$$\begin{aligned}
J_2^\e &= \int_{\R^d} E(U^\e)\,dx - \int_{\R^d} K(f^\e)\,dx + \int_{\R^d} K(f^\e)\,dx - \int_{\R^d} K(f^\e_0)\,dx\cr
&\quad + \int_{\R^d} K(f^\e_0)\,dx - \int_{\R^d} E(\bar U_0)\,dx\cr
&\leq 0 + \int_{\R^d} K(f^\e)\,dx - \int_{\R^d} K(f^\e_0)\,dx + \int_{\R^d} K(f^\e_0)\,dx - \int_{\R^d} E(\bar U_0)\,dx.
\end{aligned}$$
\noindent {\bf Estimate of $J_3^\e$}: It follows from \cite[Lemma 4.3]{KMT15} that
\[
A(U^\e|\bar U) = \begin{pmatrix}
    0   \\[4mm]
       \rho^\e (u^\e - \bar u) \otimes (u^\e - \bar u) 
\end{pmatrix}.
\]
This together with the fact $DE(\bar U) = \binom{-|\bar u|^2/2}{\bar u}$ yields
\[
J_3^\e \leq C\|\nabla_x \bar u\|_{L^\infty}\int_0^t\int_{\R^d} \rho^\e|u^\e-\bar u|^2\,dxds = C\|\nabla_x \bar u\|_{L^\infty}\int_0^t \int_{\R^d} \mh(U^\e(s)|\bar U(s))\,dxds.
\]
\noindent {\bf Estimate of $J_4^\e$}: A direct computation asserts 
\[
|J_4^\e| \leq \|\nabla_x \bar u\|_{L^\infty}\int_0^t \int_{\R^d} \lt|\int_{\R^d} (u^\e \otimes u^\e - v \otimes v)f^\e\,dv \rt|dxds.
\]
On the other hand, we get
\[
\int_{\R^d} (u^\e \otimes u^\e - v \otimes v)f^\e\,dv = \int_{\R^d} (u^\e - v) \otimes (v - u^\e)\,f^\e\,dv. 
\]
This together with \eqref{lem_energy} gives
\[
|J_4^\e| \leq C\|\nabla_x \bar u\|_{L^\infty}\max\{1,\lambda\} \e,
\]
where $C > 0$. \newline

\noindent {\bf Estimate of $J_5^\e + J_6^\e$}: Integrating by parts gives
$$\begin{aligned}
\lambda \int_0^t \int_{\R^d} \nabla_x V(x) \cdot \rho^\e(x) u^\e(x)\,dxds &= -\lambda\int_0^t \int_{\R^d} V(x) \nabla_x \cdot (\rho^\e(x,s) u^\e(x,s))\,dxds \cr
&= \lambda\int_0^t \int_{\R^d} V(x) \pa_s \rho^\e(x,s)\,dxds \cr
&= \lambda\int_{\R^d} V(x)\rho^\e(x,t)\,dx - \lambda\int_{\R^d} V(x)\rho^\e_0(x)\,dx.
\end{aligned}$$
Thus we get
\[
J_5^\e + J_6^\e = \gamma\int_0^t \int_{\R^d} \rho^\e(x) |u^\e(x)|^2\,dxds + \lambda\int_{\R^d} V(x)\rho^\e(x,t)\,dx - \lambda\int_{\R^d} V(x)\rho^\e_0(x)\,dx.
\]
\noindent {\bf Estimate of $J_7^\e$}: Note that
$$\begin{aligned}
 &\lambda\int_0^t \int_{\R^d} \rho^\e(x,s) u^\e(x,s) \cdot (\nabla_x W \star \rho^\e)(x,s)\,dxds \cr
 &\quad =  \lambda\int_0^t \int_{\R^d} \pa_s (\rho^\e(x,s) )  (W \star \rho^\e)(x,s)\,dxds\cr
 &\quad =\frac\lambda2 \int_0^t \frac{\pa}{\pa s}\lt(\int_{\R^d \times \R^d}  W(x-y) \rho^\e(x,s)   \rho^\e(y,s)\,dxdy\rt)ds\cr
 &\quad =\frac\lambda2\lt(\int_{\R^d \times \R^d}  W(x-y) \rho^\e(x,t)\rho^\e(y,t)\,dxdy - \int_{\R^d \times \R^d}  W(x-y) \rho^\e_0(x)   \rho^\e_0(y)\,dxdy \rt).
\end{aligned}$$
This yields
$$\begin{aligned}
J_7^\e &= \lambda\int_0^t \int_{\R^d} \rho^\e(x) ( u^\e(x) - \bar u(x)) \cdot (\nabla_x W \star (\bar \rho - \rho^\e))(x)\,dxds \cr
&\quad + \frac\lambda2\lt(\int_{\R^d \times \R^d}  W(x-y) \rho^\e(x,t)\rho^\e(y,t)\,dxdy - \int_{\R^d \times \R^d}  W(x-y) \rho^\e_0(x)   \rho^\e_0(y)\,dxdy \rt).
\end{aligned}$$
We now combine the estimates $J_i^\e, i=2, 5, 6,7$ to get
$$\begin{aligned}
\sum_{i \in \{2,5,6,7\}} J_i^\e &= \int_{\R^d} K(f^\e_0)\,dx - \int_{\R^d} E(\bar U_0)\,dx + \mf(f^\e) - \mf(f_0^\e) \cr
&\quad + \lambda\int_0^t \int_{\R^d} \rho^\e(x) ( u^\e(x) - \bar u(x)) \cdot (\nabla_x W \star (\bar \rho - \rho^\e))(x)\,dxds\cr
&\quad + \gamma\int_0^t \int_{\R^d} \rho^\e(x) |u^\e(x)|^2\,dxds.
\end{aligned}$$
We then use \eqref{lem_energy} and \eqref{est_uf} to find
$$\begin{aligned}
\sum_{i \in \{2,5,6,7\}} J_i^\e &\leq \int_{\R^d} K(f^\e_0)\,dx - \int_{\R^d} E(\bar U_0)\,dx \cr
&\quad + \lambda\int_0^t \int_{\R^d} \rho^\e(x) ( u^\e(x) - \bar u(x)) \cdot (\nabla_x W \star (\bar \rho - \rho^\e))(x)\,dxds.
\end{aligned}$$
We finally combine all the above estimates to conclude the proof.
\end{proof}

\begin{remark}
Note that we proved $J_4^\e = \mathcal{O}(\e)$, if $\lambda$ is a fixed constant, in contrast with \cite[Lemma 4.4]{KMT15}, where they only proved $J_4^\e = \mathcal{O}(\sqrt\e)$ due to the pressure term in the Euler equations.
\end{remark}

%
%
%
%
\subsection{Relative entropy combined with $2$-Wasserstein distance}

In this part, we show that the $2$-Wasserstein distance can be bounded by the relative entropy. 

Note that the local densities $\bar \rho$ and $\rho^\e$ satisfy
\[
\pa_t \bar \rho + \nabla_x \cdot (\bar \rho \bar u) = 0 \quad \mbox{and}  \quad \pa_t \rho^\e + \nabla_x \cdot (\rho^\e u^\e) = 0,
\]
respectively. Let us define forward characteristics $X(t) := X(t;0,x)$ and $X^\e(t) := X^\e(t;0,x)$, $t \in [0,T]$ which solve the following ODEs:
\bq\label{eq_char2}
\pa_t X(t) = \bar u(X(t),t) \quad \mbox{and}  \quad \pa_t X^\e(t) = u^\e(X^\e(t),t)
\eq
with $X(0) = X^\e(0) = x \in \R^d$, respectively. Since we assumed that $\bar u$ is bounded and Lipschitz continuous on the time interval $[0,T]$, there exists a unique solution $\bar \rho$, which is determined as the push-forward of the its initial densities through the flow maps $X$, i.e.,  $\bar\rho(t) = X(t;0,\cdot) \# \bar\rho_0$. Here $\cdot \,\# \,\cdot $ stands for the push-forward of a probability measure by a measurable map, more precisely, $\nu = \mt \# \mu$ for probability measure $\mu$ and measurable map $\mt$ implies
\[
\int_{\R^d} \varphi(y) \,d\nu(y) = \int_{\R^d} \varphi(\mt(x)) \,d\mu(x),
\]
for all $\varphi \in \mc_b(\R^d)$. Note that the solution $X(t;0,x)$ is Lipschitz in $x$ with the Lipschitz constant $e^{\|\nabla_x \bar u\|_{L^\infty}}$. Indeed, we estimate
\begin{align*}
|X(t;0,x) - X(t;0,y)| &\leq |x-y| + \int_0^t |\bar u(X(s;0,x)) - \bar u(X(s;0,y))|\,ds\cr
&\leq |x-y| + \|\nabla_x \bar u\|_{L^\infty}\int_0^t |X(s;0,x) - X(s;0,y)|\,ds.
\end{align*}
Apply Gr\"onwall's lemma to the above gives
\bq\label{est_xlip}
|X(t;0,x) - X(t;0,y)| \leq  e^{\|\nabla_x \bar u\|_{L^\infty}t}|x-y|.
\eq
On the other hand, the regularity of $u^\e$ is not enough to have the existence of solutions $X^\e(t)$ to the second differential equation in \eqref{eq_char2}. Thus, inspired by the following proposition from \cite[Theorem 8.2.1]{AGS08}, see also \cite[Proposition 3.3]{FK19}, we overcome this difficulty.
\begin{proposition}\label{prop_am}Let $T>0$ and $\rho : [0,T] \to \mathcal{P}(\R^d)$ be a narrowly continuous solution of \eqref{eq_char2}, that is, $\rho$ is continuous in the duality with continuous bounded functions, for a Borel vector field $u$ satisfying
\bq\label{est_p1}
\int_0^T\int_{\R^d} |u(x,t)|^p\rho(x,t)\,dx dt < \infty,
\eq
for some $p > 1$. Let $\Xi_T: [0,T] \to \R^d$ denote the space of continuous curves. Then there exists a probability measure $\eta$ on $\Xi_T \times \R^d$ satisfying the following properties:
\begin{itemize}
\item[(i)] $\eta$ is concentrated on the set of pairs $(\xi,x)$ such that $\xi$ is an absolutely continuous curve satisfying
\[
\dot\xi(t) = u(\xi(t),t)
\]
for almost everywhere $t \in (0,T)$ with $\xi(0) = x \in \R^d$.
\item[(ii)] $\rho$ satisfies
\[
\int_{\R^d} \varphi(x)\rho\,dx = \int_{\Xi_T \times \R^d}\varphi(\xi(t))\,d\eta(\xi,x)
\]
for all $\varphi \in \mc_b(\R^d)$, $t \in [0,T]$.
\end{itemize}
\end{proposition}
Note that it follows from \eqref{lem_energy}, see also \eqref{est_uf},  that
\[
\int_{\R^d} |u^\e|^2 \rho^\e\,dx \leq \int_{\R^d \times \R^d} |v|^2 f^\e\,dxdv  < \infty,
\]
i.e., \eqref{est_p1} holds for $p=2$, and thus by Proposition \ref{prop_am}, we have the existence of a probability measure $\eta^\e$ in $\Xi_T \times \R^d$, which is concentrated on the set of pairs $(\xi,x)$ such that $\xi$ is a solution of
\bq\label{eq_gam}
\dot{\xi}(t) = u^\e(\xi(t),t)
\eq
with $\xi(0) = x \in \R^d$. Moreover, we have
\bq\label{eq_gam2}
\int_{\R^d} \varphi(x) \rho^\e(x,t)\,dx = \int_{\Xi_T \times \R^d}\varphi(\xi(t))\,d\eta^\e(\xi,x)
\eq
for all $\varphi \in \mc_b(\R^d)$, $t \in [0,T]$.

\begin{lemma}\label{prop_rho_wa} Let $f^\e$ be the solution to the equation \eqref{main_eq} and $(\bar \rho,\bar u)$ be the strong solution to the system \eqref{main_eq2} on the time interval $[0,T]$. Then we have
\[
W_2^2(\rho^\e(t), \bar\rho(t)) \leq C\exp\lt(C\|\nabla_x \bar u\|_{L^\infty(0,T;L^\infty)}\rt)\lt(W_2^2(\rho^\e_0, \bar \rho_0) + \int_0^t \int_{\R^d} \mh(U^\e(s)|\bar U(s))\,dx\,ds\rt),
\]
for $0 \leq t \leq T$, where $\rho^\e = \int_{\R^d} f^\e\,dv$ and $C > 0$ depends only on $T$.
\end{lemma}
\begin{proof}
Let us introduce a density $\hat \rho^\e$ which is determined by the push-forward of $\rho_0^\e$ through the flow map $X$, i.e., $\hat\rho^\e = X \# \rho_0^\e$. For the proof, we estimate $W_2(\bar\rho, \hat\rho^\e)$ and $W_2(\rho^\e, \hat \rho^\e)$ to have the error estimate between $\bar\rho$ and $\rho^\e$ in $2$-Wasserstein distance.  Let us first show the estimate of $W_2(\bar\rho, \hat\rho^\e)$. We choose an optimal transport map $\mt_0^\e(x)$ between $\rho_0^\e$ and $\bar\rho_0$ such that $\rho_0^\e = \mt_0^\e \# \bar\rho_0$. Then since $\bar\rho = X \# \bar\rho_0$ and $\hat \rho^\e = X \# \rho_0^\e$, we find
\[
W_2^2(\bar\rho(t), \hat\rho^\e(t)) \leq \int_{\R^d} |X(t;0,x) - X(t;0,\mt_0^\e(x))|^2  \rho_0(x)\,dx.
\]
Then this together with the Lipschitz estimate of $X$ appeared in \eqref{est_xlip} yields
 \[
W_2^2(\bar\rho(t), \hat\rho^\e(t)) \leq  e^{2\|\nabla_x \bar u\|_{L^\infty} t}\int_{\R^d} |x - \mt_0^\e(x)|^2  \rho_0(x)\,dx = e^{2\|\nabla_x \bar u\|_{L^\infty} t}W_2^2(\bar\rho_0, \rho^\e_0),
 \]
 that is,
 \[
W_2(\bar\rho(t), \hat\rho^\e(t)) \leq  e^{\|\nabla_x \bar u\|_{L^\infty} t}W_2(\bar\rho_0, \rho^\e_0).
 \]
  For the estimate of $W_2(\rho^\e, \hat \rho^\e)$, we use the disintegration theorem of measures (see \cite{AGS08}) to write 
\[
d\eta^\e(\xi,x) = \eta^\e_x(d\xi) \otimes \rho^\e_0(x)\,dx,
\]
where $\{\eta^\e_x\}_{x \in \R^d}$ is a family of probability measures on $\Xi_T$ concentrated on solutions of \eqref{eq_gam}. We then introduce a measure $\nu^\e$ on $\Xi_T \times \Xi_T \times \R^d$ defined by
\[
d\nu^\e(\xi, x, \sigma) = \eta^\e_x(d\xi) \otimes \delta_{X(\cdot;0,x)}(d\sigma) \otimes \rho^\e_0(x)\,dx.
\]
We also introduce an evaluation map $E_t : \Xi_T  \times \Xi_T \times \R^d \to \R^d \times \R^d$ defined as $E_t(\xi, \sigma, x) = (\xi(t), \sigma(t))$. Then we readily show that measure $\pi^\e_t:= (E_t)\# \nu^\e$ on $\R^d \times \R^d$ has marginals $\rho^\e(x,t)\,dx$ and $\hat\rho^\e(y,t)\,dy$ for $t \in [0,T]$, see \eqref{eq_gam2}. This yields
\begin{align}\label{est_rho2}
\begin{aligned}
W_2^2(\rho^\e(t), \hat\rho^\e(t)) &\leq \int_{\R^d \times \R^d} |x-y|^2\,d\pi^\e_t(x,y)\cr
&=\int_{\Xi_T \times \Xi_T \times \R^d} |\sigma(t) - \xi(t) |^2 \,d\nu^\e(\xi, \sigma, x) \cr
&= \int_{\Xi_T \times \R^d} |X(t;0,x) - \xi(t)|^2 \,d\eta^\e(\xi,x).
\end{aligned}
\end{align}
In order to estimate the right hand side of \eqref{est_rho2}, we use \eqref{eq_char2} and \eqref{eq_gam} to have
\begin{align*}
&\lt|X(t;0,x) -\xi(t)\rt| \cr
&\quad = \lt|\int_0^t \bar u(X(s;0,x)) - u^\e(\xi(s),s)\,ds\rt|\cr
&\qquad \leq \int_0^t \lt|\bar u(X(s;0,x)) - \bar u(\xi(s),s)\rt|ds + \int_0^t \lt|\bar u(\xi(s),s) - u^\e(\xi(s),s)\rt|ds\cr
&\qquad \leq \|\nabla_x \bar u\|_{L^\infty}\int_0^t  \lt|X(s;0,x) - \xi(s)\rt|ds + \int_0^t \lt|\bar u(\xi(s),s) - u^\e(\xi(s),s)\rt|ds,
\end{align*}
subsequently, this yields
\[
\lt|X(t;0,x) -\xi(t)\rt| \leq Ce^{C\|\nabla_x \bar u\|_{L^\infty}}\int_0^t \lt|\bar u(\xi(s),s) - u^\e(\xi(s),s)\rt|ds,
\]
where $C>0$ is independent of $\e>0$. Combining this with \eqref{est_rho2}, we have
$$\begin{aligned}
W_2^2(\rho^\e(t), \hat\rho^\e(t)) &\leq Ce^{C\|\nabla_x \bar u\|_{L^\infty}}\int_{\Xi_T \times \R^d} \lt|\int_0^t \lt|\bar u(\xi(s),s) - u^\e(\xi(s),s)\rt|ds\rt|^2 d\eta^\e(\xi,x)\cr
&\leq Cte^{C\|\nabla_x \bar u\|_{L^\infty}}\int_0^t\int_{\Xi_T \times \R^d} \lt|\bar u(\xi(s),s) - u^\e(\xi(s),s)\rt|^2 d\eta^\e(\xi,x)\,ds\cr
&\leq Ce^{C\|\nabla_x \bar u\|_{L^\infty}}\int_0^t \int_{\R^d} |\bar u(x,s) - u^\e(x,s)|^2 \rho^\e(x,s)\,dxds\cr
&=Ce^{C\|\nabla_x \bar u\|_{L^\infty}}\int_0^t  \int_{\R^d} \mh(U^\e(s)|\bar U(s))\,dx ds,
\end{aligned}$$
where $C>0$ is independent of $\e > 0$, and we used the relation \eqref{eq_gam2}. Combining all of the above estimates asserts
\begin{align*}
W_2^2(\bar \rho(t), \rho^\e(t)) &\leq \sqrt 2 W_2^2(\bar \rho(t), \hat\rho^\e(t)) + \sqrt 2 W_2^2(\rho^\e(t), \hat\rho^\e(t))\cr
&\leq \sqrt 2 e^{2\|\nabla_x \bar u\|_{L^\infty} t}W_2^2(\bar\rho_0, \rho^\e_0) + Ce^{C\|\nabla_x \bar u\|_{L^\infty}}\int_0^t  \int_{\R^d} \mh(U^\e(s)|\bar U(s))\,dx ds\cr
&\leq Ce^{C\|\nabla_x \bar u\|_{L^\infty}}\lt( W_2^2(\bar\rho_0, \rho^\e_0) + \int_0^t  \int_{\R^d} \mh(U^\e(s)|\bar U(s))\,dx ds\rt),
\end{align*}
where $C>0$ is independent of $\e > 0$. This completes the proof.
\end{proof}

\begin{proposition}\label{prop_re2} Let $f^\e$ be the solution to the equation \eqref{main_eq} and $(\bar \rho,\bar u)$ be the strong solution to the system \eqref{main_eq2} on the time interval $[0,T]$. Then we obtain
$$
\begin{aligned}
&\int_{\R^d} \mh(U^\e(t)|\bar U(t))\,dx + (\gamma -C\lambda - e^{C_{\bar u}}(1+\lambda))\int_0^t \int_{\R^d} \mh(U^\e(s)|\bar U(s))\,dxds\cr
&\qquad \leq \int_{\R^d} \mh(U^\e_0|\bar U_0)\,dx + \int_{\R^d} \lt( \int_{\R^d} f_0^\e |v|^2\,dv  - \bar \rho_0|\bar u_0|^2\rt)dx + C_{\bar u}\max\{1,\lambda\}\e + e^{C_{\bar u}}\lambda W_2^2(\rho^\e_0,\bar \rho_0).
\end{aligned}
$$
Here $C_{\bar u} = C\|\nabla_x \bar u\|_{L^\infty(0,T;L^\infty)}$ and $C>0$ is independent of $\gamma, \lambda$ and $\e$, but depends on $T$. 
\end{proposition}
\begin{proof} Since $\nabla_x W \in \W^{1,\infty}(\R^d)$, we get
\[
\lt|\int_{\R^d} \nabla_x W(x-y)(\bar \rho(y) - \rho^\e(y))\,dy \rt| \leq CW_1(\rho^\e,\bar \rho).
\]
This enables us to estimate the last term on the right hand side of \eqref{eqn_rel} as 
$$\begin{aligned}
&\lambda\lt|\int_{\R^d} \rho^\e(x) ( u^\e(x) - \bar u(x)) \cdot (\nabla_x W \star (\bar \rho - \rho^\e))(x)\,dx \rt| \cr
&\quad \leq C\lambda W_1(\rho^\e,\bar\rho)\lt(\int_{\R^d} \rho^\e|u^\e - \bar u|^2\,dx \rt)^{1/2} \leq C\lambda W_2^2 (\rho^\e,\bar \rho) + C\lambda \int_{\R^d} \rho^\e|u^\e - \bar u|^2\,dx,
\end{aligned}$$
where we used $W_1 \leq W_2$. This together with Proposition \ref{prop_re} gives 
$$\begin{aligned}
&\int_{\R^d} \mh(U^\e(t)|\bar U(t))\,dx + (\gamma - C\lambda)\int_0^t \int_{\R^d} \rho^\e(x)| u^\e(x) - \bar u(x)|^2\,dxds\cr
&\quad \leq \int_{\R^d} \mh(U^\e_0|\bar U_0)\,dx + \int_{\R^d} \lt( \int_{\R^d} f_0^\e |v|^2\,dv  - \bar\rho_0|\bar u_0|^2\rt)dx + C\|\nabla_x \bar u\|_{L^\infty}\e\max\{1,\lambda\} \cr
&\qquad + C\lambda\int_0^t W_2^2(\rho^\e(s),\bar \rho(s)) \,ds + C\|\nabla_x \bar u\|_{L^\infty}\int_0^t \int_{\R^d} \mh(U^\e(s)|\bar U(s))\,dxds.
\end{aligned}$$
We then combine the above inequality and Lemma \ref{prop_rho_wa} to have
$$\begin{aligned}
&\int_{\R^d} \mh(U^\e(t)|\bar U(t))\,dx + (\gamma - C\lambda)\int_0^t \int_{\R^d} \rho^\e(x)| u^\e(x) - \bar u(x)|^2\,dxds\cr
&\quad \leq \int_{\R^d} \mh(U^\e_0|\bar U_0)\,dx + \int_{\R^d} \lt( \int_{\R^d} f_0^\e |v|^2\,dv  - \bar\rho_0|\bar u_0|^2\rt)dx + C_{\bar u}\e\max\{1,\lambda\} \cr
&\qquad + e^{C_{\bar u}}\lambda W_2^2(\rho^\e_0,\bar\rho_0) + e^{C_{\bar u}}(1 + \lambda)\int_0^t \int_{\R^d} \mh(U^\e(s)|\bar U(s))\,dxds,
\end{aligned}$$
where $C_{\bar u} = C\|\nabla_x \bar u\|_{L^\infty(0,T;L^\infty)}$. This completes the proof.
\end{proof}

\begin{remark}\label{rmk_hydro} If we study the hydrodynamic limit $\e \to 0$ with fixed $\gamma, \lambda > 0$, then assuming 
\[
\int_{\R^d} \mh(U^\e_0|\bar U_0)\,dx + \int_{\R^d} \lt( \int_{\R^d} f_0^\e |v|^2\,dv  - \bar\rho_0|\bar u_0|^2\rt)dx + W_2(\rho^\e_0,\bar\rho_0) = \mathcal{O}(\e)
\]
yields the relative entropy and the $2$-Wasserstein distance between solutions decays to zero as $\e \to 0$:
\[
\sup_{0 \leq t \leq T} \lt(\int_{\R^d} \rho^\e(x,t)| u^\e(x,t) - \bar u(x,t)|^2\,dx + W_2^2(\rho^\e(t),\bar \rho(t))\rt) \to 0 \quad \mbox{as} \quad \e \to 0.
\]
In this case, the limit of $f^\e$ is also determined by
\[
f^\e \rightharpoonup \bar \rho\, \delta_{v - \bar u} \quad \mbox{weakly-$*$ as} \quad \e \to 0
\]
for a.e. $t \in (0,T)$. Indeed, for $\phi \in \mc^\infty_c(\R^d \times \R^d \times [0,T])$, we have
$$\begin{aligned}
&\int_0^T \int_{\R^d \times \R^d} \lt(f^\e(x,v,t) - \bar \rho(x,t) \, \delta_{(v - \bar u(x,t))}\rt)\phi(x,v,t)\,dxdvdt\cr
&\quad = \int_0^T \int_{\R^d \times \R^d} f^\e(x,v,t) \lt(\phi(x,v,t) - \phi(x,\bar u(x,t),t) \rt) dxdvdt \cr
&\qquad +  \int_0^T \int_{\R^d} \lt( \rho^\e(x,t) - \bar\rho(x,t) \rt)\phi(x,\bar u(x,t),t)\,dxdt \cr
&=: R_1^\e + R_2^\e,
\end{aligned}$$
where $R_1^\e$ can be estimated as
\[
\lt|R_1^\e\rt| \leq C(\|\nabla_{x,v,t}\phi\|_{L^\infty})\lt(\int_0^T\int_{\R^d \times \R^d} f^\e|v - \bar u|^2\,dxdvdt \rt)^{1/2} \to 0
\]
as $\e \to 0$, due to \eqref{lem_energy}. For the estimate of $R_2^\e$, we obtain
\[
\lt| R_2^\e\rt| \leq C(\|\phi\|_{L^\infty}, \|\nabla_{x,v,t}\phi\|_{L^\infty}, \|\nabla_x \bar u\|_{L^\infty} )\int_0^T W_2(\rho^\e(t),\bar\rho(t))\,dt \to 0
\]
as $\e \to 0$. We finally note that in \cite{FK19}, $2$-Wasserstein distance is also used to handle the nonlocal velocity alignment force, however, they need a slightly stronger assumption like $\|\rho^\e_0 - \bar\rho_0\|_{L^1} = \mathcal{O}(\e)$ rather than $W_2(\rho^\e_0,\bar\rho_0) = \mathcal{O}(\e)$. We also want to emphasize that our estimate is more consistent in the sense that we need to assume the condition for $W_2(\rho^\e_0,\bar\rho_0)$ to have the estimate for $W_2(\rho^\e(t),\bar\rho(t))$. Moreover, it is not clear that the strategy used in \cite{FK19} works for the whole space case since they make use of the periodicidity and the boundedness of the domain. In a recent work \cite{CYpre}, it is observed that the $1$-Wasserstein distance can be also bounded by the relative entropy. 
\end{remark}
\begin{remark}\label{rmk_v2}  Suppose that $\gamma$ is large enough such that $\gamma -C\lambda - e^{C_{\bar u}}(1+\lambda) >0$. Then it follows from Proposition \ref{prop_re2} that
\[
\int_0^t \int_{\R^d} \mh(U^\e(s)|\bar U(s))\,dxds \leq \frac{\mathcal{I}(U^\e_0, \bar U_0) + C_{\bar u}\max\{1,\lambda\}\e + e^{C_{\bar u}}\lambda W_2^2(\rho^\e_0,\bar\rho_0)}{\gamma -C\lambda - e^{C_{\bar u}}(1+\lambda)},
\]
where 
\[
\mathcal{I}(U^\e_0, \bar U_0) = \int_{\R^d} \mh(U^\e_0|\bar U_0)\,dx + \int_{\R^d} \lt( \int_{\R^d} f_0^\e |v|^2\,dv  - \bar\rho_0|\bar u_0|^2\rt)dx.
\]
\end{remark}

We now provide the details of Proof of Proposition \ref{prop_main}.
\begin{proof}[Proof of Proposition \ref{prop_main}]Combining Lemma \ref{prop_rho_wa} and Remark \ref{rmk_v2} yields
$$\begin{aligned}
W_2^2(\rho^\e(t), \bar\rho(t)) &\leq e^{C_{\bar u}}\lt(W_2^2(\rho^\e_0, \bar\rho_0) + \int_0^t \int_{\R^d} \mh(U^\e(s)|\bar U(s))\,dxds\rt)\cr
&\leq e^{C_{\bar u}}\lt( W_2^2(\rho^\e_0, \bar\rho_0) + \frac{\mathcal{I}(U^\e_0, \bar U_0) + C_{\bar u}\max\{1,\lambda\}\e + e^{C_{\bar u}}\lambda W_2^2(\rho^\e_0,\bar \rho_0)}{\gamma -C\lambda - e^{C_{\bar u}}(1+\lambda)} \rt).
\end{aligned}$$
This concludes the desired result.
\end{proof}

%
%
%
%
\section{Proof of Theorem \ref{thm_main}: Large friction limit}\label{sec_lft}
In this section, we provide the details of proof of Theorem \ref{thm_main} on the large friction limit from the kinetic equation \eqref{main_eq} to the aggregation equation \eqref{main_conti}. Our main strategy is to combine the $2$-Wasserstein distance estimate in Proposition \ref{prop_main} and the recent work \cite{CCTpre} where the overdamped limit to the aggregation equation from damped Euler system with interaction forces is established by optimal transport techniques. We notice that the intermediate system \eqref{main_eq2} depends on the parameters $\gamma$ and $\lambda$ and the estimates in Section \ref{sec_quanti} also depend on the $\|\nabla_x \bar u\|_{L^\infty(0,T;L^\infty)}$. Thus we need to check how it depends on the parameters $\gamma$ and $\lambda$. Throughout this section, we set $\lambda=\kappa\gamma$.

\subsection{$Lip$-estimate on the velocity field.} Let us denote by $\bar u$ the strong solution to the system \eqref{main_eq2}. Our goal in this part is to provide the $L^\infty$-estimate of $\nabla_x \bar u$. 

Define the characteristic flow $\bar\eta$ associated to the fluid velocity $\bar u(x,t)$ by
\bq\label{char}
\pa_t \bar\eta(x,t) = \bar u(\bar\eta(x,t),t) \quad  \mbox{for} \quad t > 0 \quad \mbox{subject to} \quad \bar\eta(x,0) = x \in \R^d.
\eq

\begin{lemma}\label{lem_u} Let $T>0$ and $(\bar\rho,\bar u)$ be the strong solution to the system \eqref{main_eq2} on the time interval $[0,T]$. Then there exist $\gamma_* > 0$ and $\kappa_*>0$ such that 
\[
\|\nabla_x \bar u\|_{L^\infty(0,T;L^\infty)} \leq \|\nabla_x \bar u_0\|_{L^\infty} + 1
\]
for $\gamma \geq \gamma_*$ and $\kappa \leq \kappa_*$.
\end{lemma}
\begin{proof}
It follows from the momentum equations in \eqref{main_eq2} that
\[
\pa_t \nabla_x \bar u + \bar u \cdot \nabla_x^2 \bar u + (\nabla_x \bar u)^2 = - \gamma\nabla_x \bar u  -\lambda (c_V\mathbb{I}_d + \nabla_x^2 W \star \bar\rho).
\]
Then, along the characteristic flow defined in \eqref{char}, we find
$$\begin{aligned}
(\nabla_x\bar u)(\bar\eta(x,t),t) &= (\nabla_x\bar u_0)(x)e^{-\gamma t} \cr
&\quad - e^{-\gamma t}\int_0^t \lt((\nabla_x\bar u)^2(\bar\eta(x,s),s) + \lambda(c_V \mathbb{I}_d+ \nabla_x^2 W \star \bar \rho(\bar\eta(x,s),s))\rt)e^{\gamma s}\,ds,
\end{aligned}$$
and this yields
$$\begin{aligned}
\|\nabla_x\bar u(\cdot,t)\|_{L^\infty} &\leq \|\nabla_x\bar u_0\|_{L^\infty}e^{-\gamma t} + Ce^{-\gamma t}\int_0^t \lt(\|\nabla_x\bar u(\cdot,s)\|_{L^\infty}^2 + \lambda \rt)e^{\gamma s}\,ds\cr
& = \|\nabla_x\bar u_0\|_{L^\infty}e^{-\gamma t} + Ce^{-\gamma t}\int_0^t \|\nabla_x\bar u(\cdot,s)\|_{L^\infty}^2 e^{\gamma s}\,ds + \kappa(1 - e^{-\gamma t}),
\end{aligned}$$
due to $\lambda = \kappa \gamma$. Set $C_*:= \|\nabla_x \bar u_0\|_{L^\infty} + 1$ and 
\[
\mathcal{A} := \lt\{ t > 0\,:\,\|\nabla_x \bar u(\cdot,s)\|_{L^\infty} < C_*\mbox{ for } s \in [0,t) \rt\}.
\]
Since $\mathcal{A} \neq \emptyset$, we can define $T_* := \sup \mathcal{A}$, and if $T_* < T$, then the following holds:
\[
\lim_{t \to T_*\mbox{-}}\|\nabla_x \bar u(\cdot,t)\|_{L^\infty} = C_*.
\]
On the other hand, for $t < T_*$, we get
\[
\|\nabla_x \bar u(\cdot,t)\|_{L^\infty} \leq \|\nabla_x \bar u_0\|_{L^\infty} e^{-\gamma t} + \lt(\frac{CC_*^2}{\gamma} + \kappa\rt)(1 - e^{-\gamma t}).
\]
We now choose $\gamma_*$ sufficiently large and $\kappa_*$ small enough so that
$
\tfrac{CC_*^2}{\gamma} + \kappa < 1
$
for $\gamma \geq \gamma_*$ and $\kappa \leq \kappa_*$. Thus we obtain
\[
C_* = \lim_{t \to T_*\mbox{-}}\|\nabla_x \bar u(\cdot,t)\|_{L^\infty} \leq \|\nabla_x \bar u_0\|_{L^\infty} e^{-\gamma T_*} + 1 < C_*,
\]
and this is a contradiction. Hence we have $T_* \geq T$, and this completes the proof.
\end{proof}
\subsection{Overdamped limit: from Euler to aggregation equations}
Let us consider the pressureless Euler equations \eqref{main_eq2}:
\begin{align}\label{eq_wo}
\begin{aligned}
&\pa_t \rho^\gamma + \nabla_x \cdot (\rho^\gamma u^\gamma) = 0,\cr
&\pa_t (\rho^\gamma u^\gamma) + \nabla_x \cdot (\rho^\gamma u^\gamma \otimes u^\gamma) =  - \gamma \rho^\gamma \lt(u^\gamma + \kappa\lt(\nabla_x V + \nabla_x W \star \rho^\gamma\rt)\rt).
\end{aligned}
\end{align}
Then, an easy generalization of \cite[Theorem 5]{CCTpre} implies the following error estimate between $\rho^\gamma$ and $\rho$, which is a solution to \eqref{main_conti} in $2$-Wasserstein distance.
\begin{proposition}\label{prop_od} Let $T>0$ and $(\rho^\gamma,u^\gamma)$ be the strong solution of \eqref{eq_wo} for sufficiently large $\gamma > 0$, and let $(\rho, u)$ be the unique strong solution to the following equation on the time interval $[0,T]$:
\[
\pa_t \rho + \nabla_x \cdot (\rho u) = 0, \quad \rho u = -\kappa\rho(\nabla_x V + \nabla_x W \star \rho).
\]
We further assume that the initial data satisfy
\[
\mathcal{E}(\rho_0,u_0) < \infty, \quad \sup_{\gamma > 0} \mathcal{E}(\rho_0^\gamma,u_0^\gamma) < \infty, \quad \sup_{\gamma > 0}W_2(\rho_0, \rho_0^\gamma) < \infty, 
\]
and
\[
\sup_{\gamma > 0} \int_{\R^d} |u_0 - u_0^\gamma|^2 \rho_0^\gamma\,dx < \infty,
\]
where 
\[
\mathcal{E}(\rho,u) := \mathcal{E}_1(\rho,u) + \mathcal{E}_2(\rho,u):= \lt(\int_{\R^d} V \rho\,dx+\frac12\int_{\R^d \times \R^d} W(x-y)\rho(x)\rho(y)\,dxdy\rt) + \int_{\R^d} |u|^2 \rho\,dx.
\]
Then we have
\[
\int_0^T W_2^2(\rho^\gamma(t), \rho(t))\,dt  \leq \frac{M_{\gamma}}{2c_W\gamma - 1},
\]
where $M_{\gamma} > 0$ is given by
$$\begin{aligned}
M_{\gamma} &:= 4\lt(\mathcal{E}_1(\rho_0,u_0) + \mathcal{E}_1(\rho_0^\gamma,u_0^\gamma)\rt) + (1 + \gamma)W_2^2(\rho_0, \rho_0^\gamma) \cr
&\quad 
+ \frac2\gamma\lt( \mathcal{E}_2(\rho_0,u_0) + \mathcal{E}_2(\rho_0^\gamma,u_0^\gamma)\rt) + \int_{\R^d} |u_0 - u_0^\gamma|^2 \rho_0^\gamma\,dx.
\end{aligned}$$
\end{proposition}
\begin{remark}
The improvement of Proposition \ref{prop_od} with respect to \cite[Theorem 5]{CCTpre} is on the initial data assumptions to allow the initial data depending on $\gamma$.
\end{remark}
%
%
Then we are now in a position to give the details of proof of Theorem \ref{thm_main}.
\begin{proof}[Proof of Theorem \ref{thm_main}] For a given $\rho_0$ satisfying the assumptions in Theorem \ref{thm_main}, we consider its approximation $0 \leq \bar\rho_0^\e \in H^s(\R^d)$ with $s>d/2+1$ satisfying
\[
\sup_{\e > 0}\|\bar \rho_0^\e\|_{L^1} < \infty, \quad \sup_{\e > 0}\mathcal{E}(\bar\rho_0^\e,\bar u_0^\e) < \infty, \quad \mbox{and} \quad W_2^2(\rho_0, \bar\rho_0^\e) = \mathcal{O}(\e).
\]
Set $\bar u_0^\e := -\kappa (\nabla_x V + \nabla_x W \star \bar\rho_0^\e)$. Then it is clear to get
\[
W_2^2(\rho_0^\e, \bar\rho_0^\e) \leq 2W_2^2(\rho_0^\e, \rho_0) + 2W_2^2(\rho_0, \bar\rho_0^\e) = \mathcal{O}(\e) + 2W_2^2(\rho_0^\e, \rho_0)
\]
and
\[
\|\nabla_x \bar u_0^\e\|_{L^\infty} \leq C\kappa\lt(1 + \|\nabla_x^2 W\|_{L^\infty}\|\bar\rho_0^\e\|_{L^1}  \rt)   \leq C\kappa.
\]
We now take into account the pressureless Euler system \eqref{main_eq2} with above the initial data $(\bar\rho_0^\e, \bar u_0^\e)$ and the singular parameter $\gamma = 1/\e$, i.e., $\lambda = \kappa/\e$. This, together with Lemma \ref{lem_u}, Proposition \ref{prop_main}, and choosing $\e,\kappa >0$ small enough, yields
$$\begin{aligned}
W_2^2(\rho^\e(t), \bar\rho^\e(t)) &\leq e^{C\kappa}\lt( W_2^2(\rho^\e_0, \bar\rho_0^\e) + \frac{\mathcal{I}(U^\e_0, \bar U_0^\e) + C\kappa\e\lambda + e^{C\kappa}\kappa\gamma W_2^2(\rho^\e_0,\bar \rho_0^\e)}{\gamma -C\kappa\gamma - e^{C\kappa}(1+\kappa \gamma)} \rt)\cr
&=e^{C\kappa}\lt( W_2^2(\rho^\e_0, \bar\rho_0^\e) + \frac{\e\mathcal{I}(U^\e_0, \bar U_0^\e) + C\kappa^2\e + \kappa e^{C\kappa} W_2^2(\rho^\e_0,\bar \rho_0^\e)}{1 -\kappa( C+ e^{C\kappa}(1+\e))} \rt)\cr
&=\mathcal{O}(\e) + C W_2^2(\rho^\e_0, \bar\rho_0^\e),
\end{aligned}$$
where $C> 0$ is independent of $\e$ and 
\[
\mathcal{I}(U^\e_0, \bar U^\e_0) = \int_{\R^d}  \rho^\e_0(x)| u^\e_0(x) - \bar u^\e_0(x)|^2\,dx + \int_{\R^d} \lt( \int_{\R^d} f_0^\e |v|^2\,dv  - \bar \rho^\e_0|\bar u^\e_0|^2\rt)dx.
\]
Note that 
\[
\int_{\R^d} \rho^\e_0|\bar u^\e_0|^2\,dx \leq C\int_{\R^d} \rho^\e_0 V\,dx + C\|\nabla_x W \star \bar\rho^\e_0\|_{L^\infty}^2 \int_{\R^d} \rho^\e_0\,dx \leq C\lt(\int_{\R^d} \rho^\e_0 V\,dx + \|f^\e_0\|_{L^1}\|\bar\rho^\e_0\|_{L^1}^2 \rt),
\]
where $C > 0$ is independent of $\e$. Then this implies
\[
\sup_{\e > 0}\mathcal{I}(U^\e_0, \bar U^\e_0) \leq C\sup_{\e > 0}\lt(\int_{\R^d \times \R^d} f_0^\e |v|^2\,dxdv + \int_{\R^d} \rho^\e_0 V\,dx + \|f^\e_0\|_{L^1}\|\bar\rho^\e_0\|_{L^1}^2 \rt) < \infty,
\]
where $C > 0$ is independent of $\e$. Furthermore, since $W_2^2(\rho^\e_0, \bar\rho_0^\e) \leq \mathcal{O}(\e)+2W_2^2(\rho_0^\e, \rho_0)$, we have
\[
W_2^2(\rho^\e(t), \bar\rho^\e(t)) \leq \mathcal{O}(\e)+ 2W_2^2(\rho_0^\e, \rho_0).
\]
For the error estimate of solutions to \eqref{main_conti} and \eqref{main_eq2}, we use Proposition \ref{prop_od} with $\gamma = 1/\e$ to obtain
\[
\int_0^T W_2^2(\bar\rho^\e(t), \rho(t))\,dt \leq C\e + CW_2^2(\bar\rho_0^\e, \rho_0) = \mathcal{O}(\e).
\]
We finally combine all the above estimates to conclude
$$\begin{aligned}
\int_0^T W_2^2(\rho^\e(t), \rho(t))\,dt &\leq \int_0^T W_2^2(\rho^\e(t), \bar\rho^\e(t))\,dt + \int_0^T W_2^2(\bar\rho^\e(t), \rho(t))\,dt \cr
&\leq \mathcal{O}(\e) + CW_2^2(\rho_0^\e, \rho_0).
\end{aligned}$$
This completes the proof.
\end{proof}

%
%
%
%
\section{Well-posedness of equations \eqref{main_eq}, \eqref{main_conti}, and \eqref{main_eq2}}\label{sec_ext}
In this section, we show the global-in-time existence of solutions to the equations  \eqref{main_eq}, \eqref{main_conti}, and \eqref{main_eq2} under suitable assumptions on the initial data, making our main result completely rigorous. 

\subsection{Global-in-time existence of weak solutions to the equation \eqref{main_eq}}\label{sec_weak}
We first present a notion of weak solutions of the equation \eqref{main_eq} and our result on the global-in-time existence of weak solutions. 
\begin{definition}\label{def_weak} For a given $T \in (0,\infty)$, we say that $f$ is a weak solution to the equation \eqref{main_eq} if the following conditions are satisfied:
\begin{itemize}
\item[(i)] $f \in L^\infty(0,T;(L^1_+ \cap L^\infty)(\R^d \times \R^d))$,
\item[(ii)] for any $\varphi \in \mc^\infty_c(\R^d \times \R^d \times [0,T])$, 
$$\begin{aligned}
&\int_0^t \int_{\R^d \times \R^d} f(\pa_t\varphi + v \cdot \nabla_x \varphi - \lt(\gamma v + \lambda(\nabla_x V + \nabla_x W \star \rho) \rt) \cdot \nabla_v \varphi)\,dxdvds\cr
&\quad + \int_0^t \int_{\R^d \times \R^d} f(\beta(u-v) \cdot \nabla_v \varphi)\,dxdvds + \int_{\R^d \times \R^d} f_0 \varphi(\cdot,\cdot,0)\,dxdv= 0.
\end{aligned}$$
\end{itemize}
\end{definition}

We also recall the velocity averaging lemma whose proof can be found in \cite[Lemma 2.7]{KMT13}.
 
\begin{lemma}\label{lem_vel} For $1 \leq p < (d+2)/(d+1)$, let $\{G_n\}_n$ be bounded in $L^p(\R^d \times\R^d \times (0,T))$. Suppose that
\begin{itemize}
\item[(i)] $f_n$ is bounded in $L^\infty(0,T;(L^1 \cap L^\infty)(\R^d \times \R^d))$,
\item[(ii)] $(|x|^2 + |v|^2)f_n$ is bounded in $L^\infty(0,T;L^1(\R^d \times \R^d))$.
\end{itemize}
If $f_n$ and $G_n$ satisfy the following equation:
$
\pa_t f_n + v \cdot \nabla_x f_n = \nabla_v G_n, 
$
then, for any $\varphi(v)$ satisfying $\varphi(v) \leq c|v|$ as $|v|\to \infty$, the sequence 
$
\lt\{ \int_{\R^d} f_n\varphi(v)\,dv\rt\}_n
$
is relatively compact in $L^p(\R^d \times (0,T))$.
\end{lemma}

We can now show the existence results for this type of solutions.

\begin{theorem}\label{thm_weak} Let $T>0$. Suppose that $f_0$ satisfies 
\[
f_0 \in (L^1_+ \cap L^\infty)(\R^d \times \R^d) \quad \mbox{and} \quad (|v|^2 + V + W\star \rho_0)f_0 \in L^1(\R^d \times \R^d).
\]
Furthermore, we assume 
\[
V(x) = \frac{|x|^2}{2}, \quad W \mbox{ is symmetric and bounded from below}, \quad \mbox{and} \quad \nabla_x W \in L^\infty(\R^d).
\]
Then there exists a weak solution of the equation \eqref{main_eq} in the sense of Definition \ref{def_weak} satisfying
$
(|v|^2 + V + W\star \rho)f \in L^\infty(0,T;L^1(\R^d \times \R^d)).
$
Furthermore, the total energy inequality \eqref{lem_energy} holds.
\end{theorem}
For notational simplicity, in the rest of this section, we set $\beta=\lambda = \gamma = 1$. 
\begin{remark}
Our strategy can be directly applied to the case, where the confinement potential $V$ satisfies 
$0 \leq V(x) \to + \infty$ as $|x| \to +\infty$, and $|\nabla_x V(x)|^2 \ls V(x)$ for $x \in \R^d$.
Without loss of generality, we may assume that $W \geq 0$ in the rest of this subsection.
\end{remark}
The global-in-time existence of weak solutions for the Vlasov equation with local alignment forces was studied in \cite{KMT13}. In the presence of diffusion, the global-in-time existence classical solutions around the global Maxwellian was obtained in \cite{C16kin}. We basically take a similar strategy as in \cite{KMT13} and develop it to handle the additional terms, confinement and interaction potentials, in order to provide the details of proof of Theorem \ref{thm_weak}.

\subsubsection{Regularized equation}
In this part, we deal with a regularized equation of \eqref{main_eq}. Inspired by \cite{KMT13}, we regularize the local velocity $u$ and apply the high-velocity cut-off to the regularized local velocity. More precisely, we consider 
\bq\label{eq_reg}
\pa_t f + v \cdot \nabla_x f = \nabla_v \lt(f(v - \chi_\zeta(u_\delta)) + f(v + \nabla_x V + \nabla_x W \star \rho)\rt)
\eq
with the initial data
$
f(x,v,t)|_{t=0} = f_0(x,v),
$
where 
\[
\chi_\zeta(u) = u\mathbf{1}_{|u| \leq \zeta} \quad \mbox{and} \quad u_\delta := \frac{\int_{\R^d} vf\,dv}{\delta + \int_{\R^d} f\,dv} = \frac{\rho}{\delta + \rho} u
\]
with $\delta>0$ and $\zeta>0$.

Then our goal of this part is to prove the global well-posedness of the regularized equation \eqref{eq_reg}.
\begin{proposition}\label{prop_reg} Let $f_0 \geq 0$ satisfy the condition of Theorem \ref{thm_weak}. Then, for any $\delta, \zeta>0$, there exists a solution $f \in L^\infty(0,T;(L^1\cap L^p)(\R^d \times \R^d))$ with $p \in [1,\infty]$ of \eqref{eq_reg} satisfying
\bq\label{est_lp}
\|f\|_{L^\infty(0,T;L^p)} \leq e^{C/p'}\|f_0\|_{L^p} \quad \mbox{for} \quad p \in [1,\infty]
\eq
and
\[
\sup_{0 \leq t \leq T} \mf(f) + \int_0^T \int_{\R^d \times \R^d} |v|^2 f\,dxdv \leq \mf(f_0),
\]
where $C > 0$ is independent of $\delta$ and $\zeta$.
\end{proposition}
\begin{proof}Since the proof is similar to \cite[Proposition 3.1]{KMT13}, we briefly give the idea of that. 

\noindent {\bf Step 1 (Setup for fixed point argument):} We first fix $p_0 \in (1,(d+2)/(d+1))$. For a given $\bar u \in L^{p_0}(\R^d \times (0,T))$, we let $f$ be the solution of 
\bq\label{eq_reg_app}
\pa_t f + v \cdot \nabla_x f = \nabla_v \lt(f(v - \chi_\zeta(\bar u_\delta)) + f(v + \nabla_x V + \nabla_x W \star \rho)\rt)
\eq
with the initial data
\[
f(x,v,t)|_{t=0} = f_0(x,v).
\]
We then define a map $\mt$ by
\[
\bar u \mapsto \mt(\bar u) = u_\delta.
\]
\noindent {\bf Step 2 (Existence):} We first show that the operator $\mt$ is well-defined. In fact, the global-in-time existence and uniqueness of solution $f \in L^\infty(0,T;(L^1\cap L^p)(\R^d \times \R^d))$ to \eqref{eq_reg_app} is standard at this point since $\chi_\zeta(\bar u) \in L^\infty(\R^d \times (0,T))$. Furthermore, we can also obtain the uniform $L^p$ estimate \eqref{est_lp}. Indeed, it can be easily found by using the fact that
\[
\int_{\R^d} f(\chi_\zeta(\bar u) - \nabla_x V - \nabla_x W \star \rho)\cdot \nabla_v f^{p-1}\,dv = \frac1p (\chi_\zeta(\bar u) - \nabla_x V - \nabla_x W \star \rho) \cdot \int_{\R^d} \nabla_v f^p\,dv  = 0.
\]
For the energy estimate, we obtain
\bq\label{est_energy_app0}
\frac{d}{dt}\mf(f) = -2\int_{\R^d \times \R^d} |v|^2 f\,dxdv + \int_{\R^d \times \R^d} f v \cdot \chi_\zeta(\bar u)\,dxdv \leq -\int_{\R^d \times \R^d} |v|^2 f\,dxdv + \zeta^2,
\eq
and this gives
\bq\label{est_energy_app}
\sup_{0 \leq t \leq T}\mf(f(t)) + \int_0^T\int_{\R^d \times \R^d} |v|^2 f\,dxdvdt \leq \mf(f_0) + \zeta^2 T.
\eq
The continuity of the operator $\mt$ just follows from \cite[Lemma 3.3]{KMT13}. We next provide that the operator $\mt$ is compact. More precisely, let $\{ \bar u_n\}_n$ be a bounded sequence in $L^{p_0}(\R^d \times (0,T))$, then we show that $T(\bar u_n)$ converges strongly in $L^{p_0}(\R^d \times (0,T))$ up to a subsequence. This proof relies on the velocity averaging lemma, Lemma \ref{lem_vel}, and for the proof it is enough to estimate the uniform $L^q$ bound of force fields given in \eqref{eq_reg_app} with $q \leq 2$, see \cite[Section 3.2]{KMT13}. Let us denote by $G = f(v - \chi_\zeta(\bar u_\delta)) + f(v + \nabla_x V + \nabla_x W \star \rho)$. Then we find from the above $L^p$ estimate of $f$ and \eqref{est_energy_app} that
$$\begin{aligned}
\|G\|_{L^q} &\leq \zeta\|f\|_{L^q} + 2\|(x + v)f\|_{L^q} + \|(\nabla_x W\star\rho)f\|_{L^q}\cr
&\leq \zeta\|f\|_{L^q} + 4\mf(f)^{1/2} + \|\nabla_x W\|_{L^\infty}\|f\|_{L^q} <\infty,
\end{aligned}$$
where we used 
\[
\|(x + v)f\|_{L^q} \leq 2\lt( \int_{\R^d \times \R^d} (|x|^2 + |v|^2)f\,dxdv\rt)^{1/2}\|f\|_{L^{\frac{q}{2-q}}}^{1/2} \leq 2\mf(f)^{1/2}\|f\|_{L^{\frac{q}{2-q}}}^{1/2}, 
\]
for $q \leq 2$. Then using this, Lemma \ref{lem_vel}, the argument in \cite[Section 3.2]{KMT13}, we can apply Schauder fixed point theorem to conclude the existence of solutions to the regularized equation \eqref{eq_reg}.

\noindent {\bf Step 3 (Uniform energy estimate):} Similarly to \eqref{est_energy_app0}, we find
\[
\frac{d}{dt}\mf(f) = -2\int_{\R^d \times \R^d} |v|^2 f\,dxdv + \int_{\R^d \times \R^d} f v \cdot \chi_\zeta(u_\delta)\,dxdv.
\]
We then use the following facts
\[
|\chi_\zeta(u_\delta)| \leq |u_\delta| \leq |u| \quad \mbox{and} \quad \rho|u|^2 \leq \int_{\R^d} |v|^2f\,dv
\]
to get
$$\begin{aligned}
\lt|\int_{\R^d \times \R^d} f v \cdot \chi_\zeta(u_\delta)\,dxdv\rt| &\leq \lt(\int_{\R^d \times \R^d}|v|^2 f\,dxdv \rt)^{1/2}\lt(\int_{\R^d}|\chi_\zeta(u_\delta)|^2 \rho\,dx \rt)^{1/2}\cr
&\leq \int_{\R^d \times \R^d} |v|^2 f\,dxdv.
\end{aligned}$$
Hence we have
\[
\frac{d}{dt}\mf(f) \leq -\int_{\R^d \times \R^d} |v|^2 f\,dxdv.
\]
This completes the proof.
\end{proof}

\subsubsection{Proof of Theorem \ref{thm_weak}}
In order to conclude the proof of Theorem \ref{thm_weak}, we need to pass to the limits $\zeta \to +\infty$ and $\delta \to 0$. Note that we obtain the uniform $L^p$ estimate and the energy estimate in Proposition \ref{prop_reg}, and the uniform-in-$\zeta$ bound estimate of $G$ in $L^\infty(0,T;L^q(\R^d \times \R^d))$ with $q \leq 2$ can be obtained by using the similar argument as before. Those observations together with the argument in \cite[Section 4]{KMT13} conclude the proof of Theorem \ref{thm_weak}.

%
%
%
%
\subsection{Global-in-time existence of weak solutions to the equation \eqref{main_conti}}\label{sec_conti}
In this subsection, we discuss the global-in-time existence and uniqueness of weak solutions to the continuity type equation \eqref{main_conti}. We just refer to \cite{BLR,CR,BCLR,CCH, NPS01,Pou02} for related results. We adapt some of these ideas for our particular purposes. We first introduce a definition of weak solutions to the equation \eqref{main_conti} and state the our main theorem in this part.

\begin{definition}\label{def_strong3} For a given $T \in (0,\infty)$, we say that $\rho$ is a weak solution to the equation \eqref{main_conti} if the following conditions are satisfied:
\begin{itemize}
\item[(i)] $\rho \in \mc([0,T];\mathcal{P}_2(\R^d))$, 
\item[(ii)] $\rho$ satisfies the system \eqref{main_conti} in the sense of distributions.
\end{itemize}
\end{definition}

\begin{theorem}\label{thm_conti} Let $T>0$. Suppose that the confinement potential $V$ is given by $V = |x|^2/2$ and the interaction potential $W$ is symmetric and $\nabla_x W \in \W^{1,\infty}(\R^d)$. If $\rho_0 \in \mathcal{P}_2(\R^d)$, then there exists a unique global solution $\rho$ to the equation \eqref{main_conti} on the time interval $[0,T]$ in the sense of Definition \ref{def_strong3}. In particular, we have $\sqrt{\rho}(\pa_t u + u \cdot \nabla_x u)\in L^2(0,T;L^2(\R^d))$.
\end{theorem}
\begin{proof}
We first introduce the flow $\Psi: \R_+ \times \R_+ \times \R^d \to \R^d$, generated by the velocity field $u = -\nabla_x V - \nabla_x W \star \rho$:
\[
\frac{d}{dt}\Psi(t;s,x) = u(\Psi(t;s,x),t), \quad \Psi(s;s,x) = x
\]
for all $s,t \in [0,T]$. Note that the above flow is well-defined globally in time due to the regularity of $\nabla_x W \in \W^{1,\infty}$ and $\nabla_x V = x$. Concerning the integrability $\sqrt{\rho}(\pa_t u + u \cdot \nabla_x u)\in L^2(0,T;L^2(\R^d))$, we first find
\[
\|\pa_t u\|_{L^\infty} \leq \|\nabla_x W\|_{\W^{1,\infty}}\|\sqrt{\rho}u\|_{L^2} \quad \mbox{and} \quad \|\nabla_x u\|_{L^\infty} \leq C + \|\nabla_x W\|_{\W^{1,\infty}}.
\]
This yields
$$\begin{aligned}
\int_{\R^d} \rho\lt(|\pa_t u|^2 + |u|^2|\nabla_x u|^2\rt)dx &\leq \int_{\R^d} \rho |\pa_t u|^2\, dx +\|\nabla_x u\|_{L^\infty}^2\int_{\R^d} \rho |u|^2\,dx\cr
&\leq \lt(C + \|\nabla_x W\|_{\W^{1,\infty}}^2\rt)\int_{\R^d}\rho |u|^2\,dx.
\end{aligned}$$
On the other hand, it follows from \cite{CMV03, CCH, BCLR} that
\[
\int_{\R^d}\rho |u|^2\,dx \leq \int_{\R^d}\rho_0 |u_0|^2\,dx.
\]
Hence we have
\[
\int_{\R^d} \rho\lt(|\pa_t u|^2 + |u|^2|\nabla_x u|^2\rt)dx \leq C\int_{\R^d}\rho_0 |u_0|^2\,dx \leq C\lt(\int_{\R^d} \rho_0|x|^2\,dx + 1\rt).
\]
This completes the proof.
\end{proof}
%
%
%
%
\subsection{Global-in-time existence of strong solutions to the system \eqref{main_eq2}}\label{sec_strong}
In this part, we study the global-in-time existence of strong solutions to the following system:
\begin{align}\label{main_eq3}
\begin{aligned}
&\pa_t \rho + \nabla_x \cdot (\rho u) = 0, \quad (x,t) \in \R^d \times \R_+,\cr
&\pa_t (\rho u) + \nabla_x \cdot (\rho u \otimes u) = -\gamma \rho u - \lambda \rho(\nabla_x V + \nabla_x W \star \rho)
\end{aligned}
\end{align}
with the initial data
\[
(\rho(x,t),u(x,t))|_{t=0} =: (\rho_0(x), u_0(x)), \quad x \in \R^d.
\]

We now introduce a notion of strong solution to the system \eqref{main_eq3}.
\begin{definition}\label{def_strong2} Let $s > d/2+1$. For given $T\in(0,\infty)$, the pair $(\rho,u)$ is a strong solution of \eqref{main_eq3} on the time interval $[0,T]$ if and only if the following conditions are satisfied:
\begin{itemize}
\item[(i)] $\rho \in \mc([0,T];H^s(\R^d))$, $u \in \mc([0,T];Lip(\R^d)\cap L^2_{loc}(\R^d))$, and $\nabla_x^2 u \in \mc([0,T];H^{s-1}(\R^d))$,
\item[(ii)] $(\rho, u)$ satisfy the system \eqref{main_eq3} in the sense of distributions.
\end{itemize}
\end{definition}

We first present the local-in-time existence and uniqueness results for the systems \eqref{main_eq3}.

\begin{theorem}\label{thm_local}Let $s > d/2+1$ and $R>0$. Suppose that the confinement potential $V$ is given by $V = |x|^2/2$ and the interaction potential $W$ is symmetric and $\nabla_x W \in (\W^{1,1} \cap \W^{1,\infty})(\R^d)$. For any $N<M$, there is a positive constant $T^*$ depending only on $R$, $N$, and $M$ such that if 
\[
\|\rho_0\|_{H^s} + \|u_0\|_{L^2(B(0,R))}+\|\nabla_x u_0\|_{L^\infty} +  \|\nabla_x^2 u_0\|_{H^{s-1}} < N,
\]
then the Cauchy problem \eqref{main_eq3} has a unique strong solution $(\rho,u)$, in the sense of Definition \ref{def_strong2}, satisfying
\[
\sup_{0 \leq t \leq T^*}\lt(\|\rho(\cdot,t)\|_{H^s} + \|u(\cdot,t)\|_{L^2(B(0,R))} +\|\nabla_x u(\cdot,t)\|_{L^\infty} + \|\nabla_x^2 u(\cdot,t)\|_{H^{s-1}}\rt) \leq M,
\]
where $B(0,R)$ denotes a ball of radius $R$ centered at the origin.
\end{theorem}
\begin{proof} Since the proof of local-in-time existence theory is by now classical, we sketch the proof here, see \cite[Section 2.1]{CK16} for detailed discussions.  For simplicity, we set $\lambda = \gamma =1$.

\noindent {\bf Step 1 (Linearized system):} We first consider the associate linear system:
\begin{align}\label{lin_sys}
\begin{aligned}
&\pa_t \rho + \tilde u \cdot \nabla_x \rho + \rho \nabla_x \cdot \tilde u = 0,\cr
&\rho\pa_t u + \rho \tilde u \cdot \nabla_x u = - \rho u - \rho(\nabla_x V + \nabla_x W \star \rho)
\end{aligned}
\end{align}
with the initial data $(\rho_0,u_0)$ satisfying the assumptions in Theorem \ref{thm_local}. Here $\tilde u$ satisfies
\bq\label{lin_reg}
\sup_{0 \leq t \leq T}\lt(\|\tilde u(\cdot,t)\|_{L^2(B(0,R))} + \|\nabla_x \tilde u(\cdot,t)\|_{L^\infty} + \|\nabla_x^2 \tilde u(\cdot,t)\|_{H^{s-1}}\rt) \leq M.
\eq
We notice that the existence of the above linear system can be proved by a standard linear theory \cite{K73}. Since $\tilde u$ is globally Lipschitz, by using the method of characteristics, we can show the positivity of the density $\rho$. By a straightforward computation, we first find from the continuity equation in \eqref{lin_sys} that
\begin{align}\label{est_rho}
\begin{aligned}
\frac{d}{dt}\int_{\R^d} \rho^2\,dx &\leq C\|\nabla_x \tilde u\|_{L^\infty}\|\rho\|_{L^2}^2,\cr
\frac{d}{dt}\int_{\R^d} |\nabla_x \rho|^2\,dx &\leq C\|\nabla_x \tilde u\|_{L^\infty}\|\nabla_x \rho\|_{L^2}^2 + C\|\nabla_x^2 \tilde u\|_{L^2}\|\rho\|_{L^\infty}\|\nabla_x \rho\|_{L^2}.
\end{aligned}
\end{align}
For $2 \leq k \leq s$, we obtain
$$\begin{aligned}
&\frac12\frac{d}{dt}\int_{\R^d} |\nabla_x^k \rho|^2\,dx \cr
&\quad = - \int_{\R^d} \nabla_x^k \rho \cdot (\tilde u \cdot \nabla_x^{k+1} \rho)\,dx - \int_{\R^d} \nabla_x^k \rho \cdot (\nabla_x^k (\nabla_x \rho \cdot \tilde u) - \tilde u \cdot \nabla_x^{k+1} \rho)\,dx\cr
&\qquad - \int_{\R^d} \nabla_x^k \rho \cdot (\nabla_x^k (\nabla_x \cdot \tilde u)) \rho\,dx - \int_{\R^d} \nabla_x^k \rho \cdot (\nabla_x^k(\rho \nabla_x \cdot \tilde u) -  \rho\nabla_x^k (\nabla_x \cdot \tilde u))\,dx\cr
&\quad =: \sum_{i=1}^4 I_i,
\end{aligned}$$
where $\nabla_x^k$ denotes any partial derivative $\partial_x^\alpha$ with multi-index $\alpha$, $|\alpha| = k$, and we estimate
$$\begin{aligned}
I_1 &\leq \|\nabla_x \tilde u\|_{L^\infty}\|\nabla_x^k \rho\|_{L^2}^2,\cr
I_2 &\leq \|\nabla_x^k (\nabla_x \rho \cdot \tilde u) - \tilde u \cdot \nabla_x^{k+1} \rho\|_{L^2}\|\nabla_x^k\rho\|_{L^2}\cr
&\leq C\lt(\|\nabla_x^k \tilde u\|_{L^2}\|\nabla_x \rho\|_{L^\infty} +  \|\nabla_x \tilde u\|_{L^\infty}\|\nabla_x^k\rho\|_{L^2}\rt)\|\nabla_x^k\rho\|_{L^2},\cr
I_3 &\leq \|\rho\|_{L^\infty}\|\nabla_x^k\rho\|_{L^2}\|\nabla_x^{k+1}\tilde u\|_{L^2},\cr
I_4 &\leq \|\nabla_x^k(\rho \nabla_x \cdot \tilde u) -  \rho\nabla_x^k (\nabla_x \cdot \tilde u)\|_{L^2}\|\nabla_x^k\rho\|_{L^2}\cr
&\leq C\lt(\|\nabla_x^k \rho\|_{L^2}\|\nabla_x \tilde u\|_{L^\infty} + \|\nabla_x \rho\|_{L^\infty}\|\nabla_x^k \tilde u\|_{L^2} \rt)\|\nabla_x^k\rho\|_{L^2}.
\end{aligned}$$
Here, in order to bound $I_2$ and $I_4$, we used Moser-type inequality \cite[Lemma 2.1]{C16} as
\[
\|\nabla_x^k (fg) - f\nabla_x^kg\|_{L^2} \leq C\lt(\|\nabla_x f\|_{L^\infty}\|\nabla_x^{k-1}g\|_{L^2} + \|\nabla_x^k f\|_{L^2}\|g\|_{L^\infty} \rt)
\]
for $f,g \in (H^k \cap L^\infty)(\R^d)$ and $\nabla_x f \in L^\infty(\R^d)$. This, together with \eqref{lin_reg}, yields
\bq\label{est_lrho}
\frac{d}{dt}\|\rho\|_{H^s}^2 \leq CM\|\rho\|_{H^s}^2, \quad \mbox{i.e.,} \quad \sup_{0 \leq t \leq T} \|\rho(\cdot,t)\|_{H^s} \leq \|\rho_0\|_{H^s}e^{CMT},
\eq
due to $s > d/2+1$. For the estimate of $u$, we use the positivity of $\rho$ to divide the momentum equation in \eqref{lin_sys} by $\rho$ and use the similar argument as in Lemma \ref{lem_u} to get
\[
\|\nabla_x u\|_{L^\infty}e^t \leq \|\nabla_x u_0\|_{L^\infty} + CM\int_0^t \|\nabla_x u\|_{L^\infty}e^s\,ds + C(e^t - 1).
\]
Applying Gronwall's inequality to the above, we obtain
\bq\label{est_supu}
\sup_{0 \leq t \leq T}\|\nabla_x u(\cdot,t)\|_{L^\infty} \leq \|\nabla_x u_0\|_{L^\infty} e^{CMT} + C(e^{CMT} - 1).
\eq
For $2 \leq k \leq s+1$, similarly as above, we next estimate 
$$\begin{aligned}
&\frac12\frac{d}{dt}\int_{\R^d} |\nabla_x^k u|^2\,dx \cr
&\quad = - \int_{\R^d} \nabla_x^k u \cdot (\tilde u \cdot \nabla_x^{k+1} u)\,dx - \int_{\R^d} \nabla_x^k u \cdot ( \nabla_x^k(\tilde u \cdot \nabla_x u) - \tilde u \cdot \nabla_x^{k+1} u)\,dx\cr
&\qquad - \int_{\R^d} |\nabla_x^k u|^2\,dx - \int_{\R^d} \nabla_x^k u \cdot (\nabla_x^2 W \star \nabla_x^{k-1}\rho)\,dx\cr
&\leq \|\nabla_x \tilde u\|_{L^\infty}\|\nabla_x^k u\|_{L^2}^2 + C\lt(\|\nabla_x^k \tilde u\|_{L^2}\|\nabla_x u\|_{L^\infty} + \|\nabla_x \tilde u\|_{L^\infty} \|\nabla_x^k u\|_{L^2} \rt)\|\nabla_x^k u\|_{L^2}\cr
&\quad - \|\nabla_x^k u\|_{L^2}^2 + \|\nabla_x^k u\|_{L^2}\|\nabla_x^2 W\|_{L^1}\|\nabla_x^{k-1}\rho\|_{L^2}\cr
&\leq CM\|\nabla_x^k u\|_{L^2}^2 + CM\|\nabla_x u\|_{L^\infty}\|\nabla_x^k u\|_{L^2} + C\|\nabla_x^k u\|_{L^2}\|\nabla_x^{k-1}\rho\|_{L^2}.
\end{aligned}$$
Summing the above inequality over $2 \leq k \leq s+1$ gives
\[
\frac{d}{dt}\|\nabla_x^2 u\|_{H^{s-1}} \leq CM\|\nabla_x^2 u\|_{H^{s-1}}  + C\|\nabla_x \rho\|_{H^{s-1}} + CM\|\nabla_x u\|_{L^\infty}.
\]
Then we combine the above, \eqref{est_lrho}, and \eqref{est_supu} to have
\[
\frac{d}{dt}\|\nabla_x^2 u\|_{H^{s-1}} \leq CM\|\nabla_x^2 u\|_{H^{s-1}} + C\|\rho_0\|_{H^s}e^{CMT} + C\|\nabla_x u_0\|_{L^\infty} e^{CMT} + C(e^{CMT} - 1).
\]
Thus we obtain
\[
\|\nabla_x^2 u\|_{H^{s-1}} \leq \|\nabla_x^2 u_0\|_{H^{s-1}} e^{CMT}  + C(\|\rho_0\|_{H^s} + \|\nabla_x u_0\|_{L^\infty} + 1)Te^{CMT}.
\]
On the other hand, we get
$$\begin{aligned}
\frac12\frac{d}{dt}\int_{B(0,R)} |u|^2\,dx &= -\int_{B(0,R)} u \cdot ((\tilde u \cdot \nabla_x)u)\,dx - \int_{B(0,R)} |u|^2\,dx\cr
&\quad - \int_{B(0,R)} u \cdot \nabla_x V\,dx - \int_{B(0,R)} u \cdot (\nabla_x W \star \rho)\,dx\cr
&\leq \|\nabla_x u\|_{L^\infty}\|\tilde u\|_{L^2(B(0,R))}\|u\|_{L^2(B(0,R))} + C\|u\|_{L^2(B(0,R))} - \|u\|_{L^2(B(0,R))}^2\cr
&\leq C\|u\|_{L^2(B(0,R))} - \|u\|_{L^2(B(0,R))}^2,
\end{aligned}$$
due to \eqref{est_supu}. By using Gronwall's inequality, we find
\[
\frac{d}{dt}\|u\|_{L^2(B(0,R))}  \leq C - \|u\|_{L^2(B(0,R))}, \quad \mbox{i.e.,} \quad \|u\|_{L^2(B(0,R))} \leq \|u_0\|_{L^2(B(0,R))} + C(e^T - 1).
\]
Combining all of the above observations yields
$$\begin{aligned}
&\|\rho\|_{H^s} + \|u\|_{L^2(B(0,R))}+ \|\nabla_x u\|_{L^\infty} + \|\nabla_x^2 u\|_{H^{s-1}} \cr
&\quad \leq (\|\rho_0\|_{H^s} + \|u_0\|_{L^2(B(0,R))} +\|\nabla_x^2 u_0\|_{H^{s-1}}) e^{CMT}\cr
&\qquad  + C(\|\rho_0\|_{H^s} + \|\nabla_x u_0\|_{L^\infty} + 1)Te^{CMT} + C(e^{CMT} - 1)\cr
&\quad \leq (N + (N + 1)T)e^{CMT} + C(e^{CMT} - 1).
\end{aligned}$$
We finally choose $T^* >0$ small enough such that the right hand side of the above inequality is less than $M$. Hence we have
\[
\sup_{0 \leq t \leq T^*}\lt(\|\rho(\cdot,t)\|_{H^s} + \|u(\cdot,t)\|_{L^2(B(0,R))}+ \|\nabla_x u(\cdot,t)\|_{L^\infty} + \|\nabla_x^2 u(\cdot,t)\|_{H^{s-1}}\rt) \leq M.
\]
Notice that $T^*$, $N$, and $M$ do not depend on $\tilde u$. \newline

\noindent {\bf Step 2 (Existence):} We now construct the approximated solutions $(\rho^n,u^n)$ for the system \eqref{main_eq3} by solving the following linear system:
$$\begin{aligned}
&\pa_t \rho^{n+1} + u^n \cdot \nabla_x \rho^{n+1} + \rho^{n+1} \nabla_x \cdot u^n = 0,\cr
&\rho^{n+1}\pa_t u^{n+1} + \rho^{n+1} u^n \cdot \nabla_x u^{n+1} = - \rho^{n+1}u^{n+1} - \rho^{n+1}(\nabla_x V + \nabla_x W \star \rho^{n+1}),
\end{aligned}$$
with the initial data and first iteration step defined by
\[
(\rho^n(x,0),u^n(x,0))=(\rho_0(x),u_0(x)) \quad \mbox{for all} \quad n \geq 1, \quad x\in \R^d,
\]
and
\[
(\rho^0(x,t),u^0(x,t)) = (\rho_0(x),u_0(x)), \quad (x,t) \in \R^d \times \R_+.
\]
Then it follows from {\bf Step 1} that for any $N < M$, there exists $T^* > 0$ such that if $\|\rho_0\|_{H^s} + \|u_0\|_{L^2(B(0,R))} + \|\nabla_x u_0\|_{L^\infty} +  \|\nabla_x^2 u_0\|_{H^{s-1}} < N$, then we have
\[
\sup_{n \geq 0} \sup_{0 \leq t \leq T^*}\lt(\|\rho^n(\cdot,t)\|_{H^s} + \|u^n(\cdot,t)\|_{L^2(B(0,R))} + \|\nabla_x u^n(\cdot,t)\|_{L^\infty} + \|\nabla_x^2 u^n(\cdot,t)\|_{H^{s-1}}\rt) \leq M.
\]
Note that $\rho^{n+1} - \rho^n$ and $u^{n+1} - u^n$ satisfy
$$\begin{aligned}
&\pa_t (\rho^{n+1} - \rho^n) + (u^n - u^{n-1})\cdot \nabla_x \rho^{n+1} + u^{n-1} \cdot \nabla_x (\rho^{n+1} - \rho^n) \cr
&\qquad + (\rho^{n+1} - \rho^n) \nabla_x \cdot u^n + \rho^n \nabla_x \cdot (u^n - u^{n-1}) = 0
\end{aligned}$$
and
$$\begin{aligned}
&\pa_t (u^{n+1} - u^n) + (u^n - u^{n-1})\cdot \nabla_x u^{n+1} + u^{n-1} \cdot \nabla_x (u^{n+1} - u^n) \cr
&\qquad = - (u^{n+1} - u^n) - \nabla_x W \star (\rho^{n+1} - \rho^n),
\end{aligned}$$
respectively. Then a straightforward computation gives
\[
\|(\rho^{n+1} - \rho^n)(\cdot,t)\|_{L^2}^2 \leq C\int_0^t \lt(\|(\rho^{n+1} - \rho^n)(\cdot,s)\|_{L^2}^2 + \|(u^n - u^{n-1})(\cdot,s)\|_{H^1}^2\rt)ds,
\]
where $C > 0$ depends on $\|\nabla_x \rho^{n+1}\|_{L^\infty}$, $\|\nabla_x u^n\|_{L^\infty}$, $\|\nabla_x u^{n+1}\|_{L^\infty}$, and $\|\rho^n\|_{L^\infty}$.
We also find
\[
\|(u^{n+1} - u^n)(\cdot,t)\|_{H^1}^2 \leq C\int_0^t \lt(\|(\rho^{n+1} - \rho^n)(\cdot,s)\|_{L^2}^2 + \|(u^n - u^{n-1})(\cdot,s)\|_{H^1}^2\rt)ds,
\]
where $C > 0$ depends on $\|\nabla_x u^{n+1}\|_{\W^{1,\infty}}$, $\|\nabla_x u^{n-1}\|_{\W^{1,\infty}}$, and $\|\nabla_x W\|_{\W^{1,1}}$. This provides that $(\rho^n,u^n)$ is a Cauchy sequence in $\mc([0,T];L^2(\R^d)) \times \mc([0,T];H^1(\R^d))$. Interpolating this strong convergences with the above uniform-in-$n$ bound estimates gives
\[
\rho^n \to \rho \quad \mbox{in }\mc([0,T_*]; H^{s-1}(\R^d)), \quad u^n \to u \quad \mbox{in }\mc([0,T_*]; H^1(B(0,R))) \quad \mbox{as } n\to\infty,
\]
\[
\nabla_x u^n \to \nabla_x u \quad \mbox{in } \mc(\R^d \times [0,T_*]), \quad \mbox{and} \quad \nabla_x^2 u^n \to \nabla_x^2 u \quad \mbox{in } \mc([0,T_*];H^{s-2}(\R^d)) \quad \mbox{as } n\to\infty,
\]
due to $s > d/2+1$. In order to show the limiting functions $\rho$ and $u$ satisfy the regularity in Theorem \ref{thm_local} we can use a standard functional analytic arguments. For more details, we refer to \cite[Section 2.1]{CK16} and \cite[Appendix A]{CCZ16}. We also notice that it is easy to show the limiting functions $\rho$ and $u$ are solutions to \eqref{main_eq3} in the sense of Definition \ref{def_strong2}.

\noindent {\bf Step 3 (Uniqueness):} Let $(\rho_1,u_1)$ and $(\rho_2,u_2)$ be the strong solutions obtained in the previous step with the same initial data $(\rho_0,u_0)$. Then it directly follows from the Cauchy estimate in {\bf Step 2} that
\[
\|(\rho_1 - \rho_2)(\cdot,t)\|_{L^2}^2 + \|(u_1 - u_2)(\cdot,t)\|_{H^1}^2 \leq C\int_0^t \lt(\|(\rho_1 - \rho_2)(\cdot,s)\|_{L^2}^2 + \|(u_1 - u_2)(\cdot,s)\|_{H^1}^2\rt)ds.
\]
Thus we obtain
\[
\|(\rho_1 - \rho_2)(\cdot,t)\|_{L^2}^2 + \|(u_1 - u_2)(\cdot,t)\|_{H^1}^2 \equiv 0
\]
for all $t \in [0,T^*]$. Hence we have the uniqueness of strong solutions.
\end{proof}

We next show global-in-time existence of strong solutions to the system \eqref{main_eq3} under additional assumptions on parameters $\gamma$, $\lambda$, see below, and the interaction potential $W$. We remark that the assumption on $\gamma$ and $\lambda$ is used in Lemma \ref{lem_u} for the uniform bound estimate of $\nabla_x u$. The strong regularity of $\nabla_x W$ is needed for the global-in-time existence of solutions. Note that we do not require any small assumptions on the initial data.

\begin{theorem}\label{thm_glo} Let $s> d/2 + 1$, $T>0$, and $R>0$. Suppose that the confinement potential $V$ is given by $V = |x|^2/2$ and the interaction potential $W$ is symmetric and $\nabla_x W \in (\W^{1,1} \cap \W^{[d/2]+1,\infty})(\R^d)$. Suppose that initial data $(\rho_0, u_0)$ satisfy 
\[
\rho_0 \in H^s(\R^d), \quad u_0 \in (Lip \cap L^2_{loc})(\R^d), \quad \mbox{and} \quad \nabla_x^2 u_0 \in H^{s-1}(\R^d).
\]
Then there exist $\gamma_* > 0$ and $\kappa_*>0$ such that 
\[
\sup_{0 \leq t \leq T}\lt( \|\rho(\cdot,t)\|_{H^s} + \|u(\cdot,t)\|_{L^2(B(0,R))}+\|\nabla_x u(\cdot,t)\|_{L^\infty} + \|\nabla_x^2 u(\cdot,t)\|_{H^{s-1}} \rt) \leq C
\]
for $\gamma \geq \gamma_*$ and $\kappa \leq \kappa_*$, where $C$ depends on the initial data $(\rho_0, u_0)$, $T$, $\gamma_*$, $\kappa_*$, and $\|\nabla_x W\|_{\W^{[d/2]+1,1}}$. Here $\gamma_* $ and $\kappa_*$ are appeared in Lemma \ref{lem_u}.
\end{theorem}
\begin{proof}Similarly as \eqref{est_rho}, we estimate
\[
\frac{d}{dt}\|\rho\|_{H^1}^2 \leq C\|\nabla_x u\|_{L^\infty}\|\rho\|_{H^1}^2 + C\|\nabla_x^2 u\|_{L^2}\|\rho\|_{L^\infty}\|\nabla_x \rho\|_{L^2}
\]
and
\[
\frac{d}{dt}\|\nabla_x^2 u\|_{L^2}^2 \leq  C\|\nabla_x u\|_{L^\infty}\|\nabla_x^2 u\|_{L^2}^2 + C\|\nabla_x^2 u\|_{L^2}\|\rho\|_{H^1}.
\]
This yields
\[
\frac{d}{dt}\lt(\|\rho\|_{H^1}^2 + \|\nabla_x^2 u\|_{L^2}^2 \rt) \leq C\lt(\|\nabla_x u\|_{L^\infty} + \|\rho\|_{L^\infty} +1\rt)\lt(\|\rho\|_{H^1}^2 + \|\nabla_x^2 u\|_{L^2}^2 \rt),
\]
i.e.,
$$\begin{aligned}
&\sup_{0 \leq t \leq T} \lt(\|\rho(\cdot,t)\|_{H^1} + \|\nabla_x^2 u(\cdot,t)\|_{L^2}\rt) \cr
&\qquad \leq C\lt(\|\rho_0\|_{H^1} + \|\nabla_x^2 u_0\|_{L^2} \rt) \exp\lt( \int_0^T \|\nabla_x u(\cdot,t)\|_{L^\infty} + \|\rho(\cdot,t)\|_{L^\infty} \,dt \rt).
\end{aligned}$$
On the other hand, it follows from Lemma \ref{lem_u} that there exist $\gamma_* > 0$ and $\kappa_*>0$ such that 
\[
\sup_{0 \leq t \leq T}\|\nabla_x u(\cdot,t)\|_{L^\infty} \leq \|\nabla_x u_0\|_{L^\infty} + 1
\]
for $\gamma \geq \gamma_*$ and $\kappa \leq \kappa_*$. Then this, together with using the method of characteristics, gives
\[
\sup_{0 \leq t \leq T}\|\rho(\cdot,t)\|_{L^\infty} \leq \|\rho_0\|_{L^\infty} \exp\lt(\int_0^T\|\nabla_x u(\cdot,t)\|_{L^\infty}\,dt\rt) \leq C\|\rho_0\|_{L^\infty}\exp\lt( \|\nabla_x u_0\|_{L^\infty} + 1 \rt).
\]
Combining all of the above observations, we obtain
\begin{align}\label{est_f1}
\begin{aligned}
&\sup_{0 \leq t \leq T}\lt( \|\rho(\cdot,t)\|_{H^1} + \|\nabla_x u(\cdot,t)\|_{L^\infty} + \|\nabla_x^2 u(\cdot,t)\|_{L^2} \rt)\cr
&\qquad \leq C\exp\lt(\|\rho_0\|_{H^1} + \|\nabla_x u_0\|_{L^\infty} + \|\nabla_x^2 u_0\|_{L^2} \rt)
\end{aligned}
\end{align}
for $\gamma \geq \gamma_*$ and $\kappa \leq \kappa_*$. We also easily estimate
\[
\sup_{0 \leq t \leq T}\|u(\cdot,t)\|_{L^2(B(0,R))} \leq C\|u_0\|_{L^2(B(0,R))}.
\]
For $0 \leq k \leq s$, we find
\bq\label{est_rho1}
\frac{d}{dt} \|\nabla_x^k \rho\|_{L^2} \leq C\|\nabla_x^k \rho\|_{L^2} + C\|\nabla_x^k u\|_{L^2}\|\nabla_x \rho\|_{L^\infty} + C\|\nabla_x^{k+1}u\|_{L^2}
\eq
and
\bq\label{est_u1}
\frac{d}{dt} \|\nabla_x^{k+1} u\|_{L^2} \leq C\|\nabla_x^{k+1}u\|_{L^2} + C\|\nabla_x^{k+1}(\nabla_x W \star \rho)\|_{L^2}.
\eq
Then we have from \eqref{est_u1}
$$\begin{aligned}
\frac{d}{dt} \|\nabla_x^2 u\|_{H^{[d/2]}} &\leq C\|\nabla_x^2 u\|_{H^{[d/2]}} + C\sum_{1 \leq k \leq [d/2]+1}\|\nabla_x^{k+1}(\nabla_x W \star \rho)\|_{L^2}\cr
&\leq C\|\nabla_x^2 u\|_{H^{[d/2]}} + C\sum_{1 \leq k \leq [d/2]+1}\|\nabla_x^{k+1}W\|_{L^1}\|\nabla_x \rho\|_{L^2}\cr
&\leq C\|\nabla_x^2 u\|_{H^{[d/2]}} + C\|\nabla_x W\|_{\W^{[d/2]+
1,1}}\|\nabla_x \rho\|_{L^2}.
\end{aligned}$$
This together with \eqref{est_f1} implies
\bq\label{est_u2}
\sup_{0 \leq t \leq T}\|\nabla_x^2 u(\cdot,t)\|_{H^{[d/2]}} \leq C,
\eq
where $C$ depends on the initial data $(\rho_0, u_0)$, $\nabla_x W$, and $T$. We back to \eqref{est_rho1} to obtain
$$\begin{aligned}
\frac{d}{dt} \|\nabla_x^2 \rho\|_{H^{[d/2]}} &\leq C \|\nabla_x^2 \rho\|_{H^{[d/2]}} + C\|\nabla_x^2 u\|_{H^{[d/2]}}\|\nabla_x \rho\|_{L^\infty} + C\|\nabla_x^3 u\|_{H^{[d/2]}}\cr
&\leq C \|\nabla_x^2 \rho\|_{H^{[d/2]}} + C\|\nabla_x \rho\|_{H^{[d/2]+1}} + C\|\nabla_x^3 u\|_{H^{[d/2]}}\cr
&\leq C + C\|\nabla_x^2 \rho\|_{H^{[d/2]}} + C\|\nabla_x^3 u\|_{H^{[d/2]}},
\end{aligned}$$
where we used \eqref{est_u2} and \eqref{est_f1}. It also follows from \eqref{est_u1} that
$$\begin{aligned}
\frac{d}{dt}\|\nabla_x^3 u\|_{H^{[d/2]}} &\leq C\|\nabla_x^3 u\|_{H^{[d/2]}} + \|\nabla_x W\|_{\W^{[d/2]+ 1,1}}\|\nabla_x^2 \rho\|_{L^2}\cr
&\leq C\|\nabla_x^3 u\|_{H^{[d/2]}} + \|\nabla_x^2 \rho\|_{H^{[d/2]}}.
\end{aligned}$$
Combining the above two differential inequalities yields
\[
\frac{d}{dt}\lt(\|\nabla_x^2 \rho\|_{H^{[d/2]}} + \|\nabla_x^3 u\|_{H^{[d/2]}} \rt) \leq C + C\lt(\|\nabla_x^2 \rho\|_{H^{[d/2]}} + \|\nabla_x^3 u\|_{H^{[d/2]}} \rt),
\]
and subsequently we find
\[
\sup_{0 \leq t \leq T}\lt(\|\rho(\cdot,t)\|_{H^{[d/2]+2}} + \|\nabla_x^2 u(\cdot,t)\|_{H^{[d/2]+1}}\rt) \leq C,
\]
where $C > 0$ depends on the initial data $\|\rho_0\|_{H^{[d/2]+2}}$, $\|\nabla_x^2 u_0\|_{H^{[d/2]+1}}$, $\|\nabla_x u_0\|_{L^\infty}$, $\|\nabla_x W\|_{\W^{[d/2]+ 1,1}}$, and $T$. We next estimate \eqref{est_rho1} and \eqref{est_u1} as 
\[
\frac{d}{dt}\lt(\|\nabla_x^k \rho\|_{L^2} + \|\nabla_x^{k+1} u\|_{L^2}\rt) \leq C\lt(\|\nabla_x^k \rho\|_{L^2} + \|\nabla_x^{k+1} u\|_{L^2}\rt) + C\|\nabla_x^k u\|_{L^2},
\]
where we used
\[
\|\nabla_x^{k+1}(\nabla_x W \star \rho)\|_{L^2} \leq \|\nabla_x^2 W\|_{L^1}\|\nabla_x^k \rho\|_{L^2} \quad \mbox{and} \quad \|\nabla_x \rho\|_{L^\infty} \leq C\|\nabla_x \rho\|_{H^{[d/2]+1}} \leq C.
\]
By summing it over $2 \leq k \leq s$ and applying Gronwall's inequality to the resulting differential inequality, we finally have
\[
\sup_{0 \leq t \leq T}\lt(\|\nabla_x^2 \rho(\cdot,t)\|_{H^{s-2}} + \|\nabla_x^2 u(\cdot,t)\|_{H^{s-1}}\rt) \leq C,
\]
where $C > 0$ depends on the initial data $\|\rho_0\|_{H^s}$, $\|\nabla_x^2 u_0\|_{H^{s-1}}$, $\|\nabla_x u_0\|_{L^\infty}$, $\|\nabla_x W\|_{\W^{[d/2]+ 1,1}}$, and $T$. Combining all of the above discussion concludes the desired result.
\end{proof}


\section*{Acknowledgements}
JAC was partially supported by the EPSRC grant number EP/P031587/1. YPC was supported by NRF grant(No. 2017R1C1B2012918 and 2017R1A4A1014735) and POSCO Science Fellowship of POSCO TJ Park Foundation. 

%
%
%
%

\end{document}